\documentclass[final,leqno]{siamltex}

\usepackage{amsfonts}
\usepackage{amsmath}
\DeclareMathOperator*{\esssup}{ess\,sup}
\DeclareMathOperator*{\argmax}{arg\,max}
\newtheorem{assumption}[theorem]{Assumption}
\newtheorem{remark}[theorem]{Remark}

\title{Stochastic Control Representations for
Penalized Backward Stochastic Differential Equations\thanks{The work
is partially supported by a start-up research fund from King's
College London, and the Oxford-Man Institute, University of
Oxford.}}

\author{Gechun Liang\thanks{Department of Mathematics, King's College London, London, WC2R 2LS,
U.K.. Email adress: {\tt gechun.liang@kcl.ac.uk}.}}

\begin{document}

\maketitle

\begin{abstract}
This paper shows that penalized backward stochastic differential
equation (BSDE), which is often used to approximate and solve the
corresponding reflected BSDE, admits both optimal stopping
representation and optimal control representation. The new feature
of the optimal stopping representation is that the player is allowed
to stop at exogenous Poisson arrival times. The convergence rate of
the penalized BSDE then follows from the optimal stopping
representation. The paper then applies to two classes of equations,
namely multidimensional reflected BSDE and reflected BSDE with a
constraint on the hedging part, and gives stochastic control
representations for their corresponding penalized equations.
\end{abstract}



\pagestyle{myheadings} \thispagestyle{plain} \markboth{Gechun
Liang}{Stochastic Control Representations for Penalized BSDEs}

\section{Introduction}

El Karoui et al \cite{ElKaroui19971} introduced penalized backward
stochastic differential equation (\emph{penalized BSDE} for short)
to solve reflected backward stochastic differential equation
(\emph{reflected BSDE} for short), and they showed that the solution
of a reflected BSDE corresponds to the value of a nonlinear optimal
stopping time problem. In this paper, our main result is to show
that the solution of the associated penalized BSDE also corresponds
to the value of some nonlinear optimal stopping time problem, and
the parameter $\lambda$ appearing in the penalized equation is
nothing but the intensity of some exogenous Poisson process.

Let $(W_t)_{t\geq 0}$ be a $d$-dimensional standard Brownian motion
defined on a filtered probability space
$(\Omega,\mathcal{F},\mathbb{F}=\{\mathcal{F}_t\}_{t\geq
0},\mathbf{P})$ satisfying the \emph{usual conditions}, i.e. the
filtration $\mathbb{F}$ is right continuous and complete.
In El Karoui et al \cite{ElKaroui19971}, the authors introduced the
following reflected BSDE
\begin{equation}\label{RBSDE1}
Y_t=\xi+\int_t^{T}f_s(Y_s,Z_s)ds+\int_t^{T}dK_s-\int_t^{T}Z_sdW_s
\end{equation}
under the constraints
\begin{align*}
\text{(Dominating Condition)}:\ \ \ & Y_t\geq S_t\ \text{for}\ t\in[0,T],\\
\text{(Skorohod Condition)}:\ \ \ & \int_0^T(Y_t-S_t)dK_t=0\
\text{for}\ K\ \text{continuous\ and\ increasing},
\end{align*}
where the terminal data $\xi$, the driver $f_s(y,z)$, and the
obstacle $(S_t)_{0\leq t\leq T}$ are the given data for the
equation. A solution to the reflected BSDE (\ref{RBSDE1}) is a
triplet of $\mathbb{F}$-adapted processes $(Y,Z,K)$, where $Z$ is a
kind of \emph{hedging} process, and $K$ is a kind of \emph{local time}
process. The equation (\ref{RBSDE1}) corresponds to a backward
Skorohod problem, which in turn gives the \emph{local time} process $K$
a Skorohod representation. See Qian and Xu \cite{Qian} in this
direction.

On the other hand, as shown in \cite{ElKaroui19971}, (\ref{RBSDE1})
also has an interesting interpretation in the sense that its
solution is the value of a nonlinear optimal stopping time problem:
For any time $t\in[0,T]$, the value of the following optimal
stopping time problem
\begin{equation}\label{OptimalStopping1}
y_t=\esssup_{\tau\in\mathcal{R}(t)}\mathbf{E}\left[\int_t^{\tau\wedge
T}f_s(Y_s,Z_s)ds+S_{\tau}\mathbf{1}_{\{\tau<
T\}}+\xi\mathbf{1}_{\{\tau\geq T\}}|\mathcal{F}_t\right],
\end{equation}
where the control set $\mathcal{R}(t)$ is defined as
$$\mathcal{R}(t)=\{\mathbb{F}\text{-stopping\ time}\ \tau\ \text{for}\ t\leq \tau\leq T\},$$
is given by the solution to the reflected BSDE (\ref{RBSDE1}):
$y_t=Y_t$ $a.s.$. The optimal stopping time is given by
$\tau^*_t=\inf\{s\geq t: Y_s=S_s\}\wedge T$. The nonlinear optimal
stopping problem (\ref{OptimalStopping1}) is closely related to
pricing and hedging American options as shown in El Karoui et al
\cite{ElKaroui19972}.

One way to solve the reflected BSDE (\ref{RBSDE1}) is to iterate the
solution of the corresponding backward Skorohod problem by Picard
iteration. The other way, which seems more commonly used in the
literature, is to approximate the \emph{local time} process $K$ by
$$K^{\lambda}_t=\int_0^t\lambda\max\{0,S_s-Y^{\lambda}_s\}ds,$$
where $(Y^{\lambda},Z^{\lambda})$ is the solution of the following
penalized BSDE
\begin{equation}\label{penalizedBSDE1}
Y_t^{\lambda}=\xi+\int_t^{T}f_s(Y_s^{\lambda},Z_s^{\lambda})ds
+\int_t^T\lambda\max\{0,S_s-Y^{\lambda}_s\}ds-\int_t^{T}Z_s^{\lambda}dW_s.
\end{equation}
Under Assumption \ref{Assumption} introduced below, El Karoui et al
\cite{ElKaroui19971} proved that $Y^{\lambda}$ is increasing in
$\lambda$, and
\begin{equation}\label{Convergence}
\lim_{\lambda\uparrow\infty}\mathbf{E}\left[\sup_{t\in[0,T]}|Y_t^{\lambda}-Y_t|^2+\int_0^T|Z_t^{\lambda}-Z_t|^2dt+
\sup_{t\in[0,T]}|K_t^{\lambda}-K_t|^2\right]=0.
\end{equation}

Our aim is to give stochastic control representations for the
penalized BSDE (\ref{penalizedBSDE1}). Our main result is to prove
that the penalized BSDE (\ref{penalizedBSDE1}) also admits an
optimal stopping representation, which will in turn converge to the
original optimal stopping time problem (\ref{OptimalStopping1}) with convergence rate $\frac{1}{\lambda}$
(see (\ref{relation1}) and (\ref{relation2}) below).

We impose the following standard assumption on the data set
$(\xi,f,S)$ as in El Karoui \cite{ElKaroui19971}, so that both
(\ref{RBSDE1}) and (\ref{penalizedBSDE1}) admit unique solutions.
\begin{assumption}\label{Assumption}
\begin{itemize}
\item The terminal data $\xi$ is $\mathbb{L}^2$-square integrable:
$\mathbf{E}[|\xi|^2]<\infty;$
\item The driver $f:
\Omega\times[0,T]\times\mathbb{R}\times\mathbb{R}^d\rightarrow\mathbb{R}
$ is uniformly Lipschitz continuous:
$$|f_t(y,z)-f_t(\bar{y},\bar{z})|\leq C(|y-\bar{y}|+|z-\bar{z}|)\ a.s.\ \text{for\ some}\ C>0,$$
with $f_t(0,0)$ being $\mathbb{F}$-adapted and $\mathbb{H}^2$-square
integrable: $\mathbf{E}\left[\int_0^T|f_t(0,0)|^2dt\right]<\infty$;
\item The obstacle process $S$ is a continuous $\mathbb{F}$-adapted
process, and uniformly square integrable:
$\mathbf{E}\left[\sup_{t\in[0,T]}|S_t|^2\right]<\infty$.
\end{itemize}
\end{assumption}

In fact, the above conditions could be relaxed. See, for example,
Peng and Xu \cite{Xu} and Lepeltier and Xu \cite{Xu1} extending to
RCLL obstacles, and Kobylanski et al \cite{Kobylanski} and Bayraktar
and Song \cite{Bayraktar} among others extending to the driver
$f_s(y,z)$ with quadratic growth in $z$. However, we only stick with
the above standard assumption in this paper. Under the above
standard assumption, we have the following representation which is
the main result of this paper.

Let $\{T_n\}_{n\geq 0}$ be the arrival times of an independent
Poisson process with intensity $\lambda$ and minimal augmented
filtration $\{\mathcal{H}_t\}_{t\geq 0}$. Define
$\mathcal{G}_t=\mathcal{F}_t\vee\mathcal{H}_t$ and
$\mathbb{G}=\{\mathcal{G}_t\}_{t\geq 0}$. Since $T_0=0$ and
$T_{\infty}=\infty$, there exists an integer-valued random variable
$M<\infty$ such that $T_M\leq T< T_{M+1}$, i.e.
$M(\omega)=\sum_{n\geq 0}n\mathbf{1}_{\{T_n(\omega)\leq T<
T_{n+1}(\omega)\}}$.


\begin{theorem}\label{Theorem}Suppose that Assumption \ref{Assumption}
holds. Denote $(Y^{\lambda},Z^{\lambda})$ as the unique solution to
the penalized BSDE (\ref{penalizedBSDE1}). For any integer $i\geq
1$, define the control set $\mathcal{R}_{T_{i}}{(\lambda)}$ as
$$\mathcal{R}_{T_i}{(\lambda)}=\left\{\mathbb{G}\text{-stopping\ time}\ \tau\ \text{for}\ \tau(\omega)=T_N(\omega)
\ \text{where}\ i\leq N\leq M+1.\right\}$$
Then conditional on $\{T_{i-1}\leq t<T_{i}\}$, the value of the
following optimal stopping time problem
\begin{equation}\label{OptimalStopping2}
y_t^{\lambda}=\esssup_{\tau\in\mathcal{R}_{T_i}{(\lambda)}}\mathbf{E}\left[\int_t^{\tau\wedge
T}f_s(Y_s^{\lambda},Z_s^{\lambda})ds+S_{\tau}\mathbf{1}_{\{\tau<
T\}}+\xi\mathbf{1}_{\{\tau\geq T\}}|\mathcal{G}_t\right]
\end{equation}
is given by the solution to the penalized BSDE
(\ref{penalizedBSDE1}): $y_t^{\lambda}=Y_t^{\lambda}$ $a.s.$. The
optimal stopping time is given by $\tau^*_{T_i}=\inf\{T_{N}\geq T_i:
Y^{\lambda}_{T_{N}}\leq S_{T_{N}}\}\wedge T_{M+1}.$
\end{theorem}

Note that on $\{T_{i-1}<t<T_i\}$, there exists an
$\mathcal{F}_t$-measurable random variable $\tilde{y}^{\lambda}_t$
such that $\tilde{y}^{\lambda}_t=y^{\lambda}_t$, so $y_t^{\lambda}$
can also be regarded as $\mathcal{F}_t$-measurable in this
situation. On the other hand, the subscript $T_i$ in
$\mathcal{R}_{T_i}(\lambda)$ represents the smallest stopping time
that is allowed to choose, and $\lambda$ represents the intensity of
the underlying Poisson process.

There are two new features of the optimal stopping time problem
(\ref{OptimalStopping2}): First, there is a control constraint in
the sense that only stopping at Poisson arrival times is allowed;
Secondly, the player is not allowed to stop at the initial starting
time $t$. By the convergence (\ref{Convergence}) and Theorem
\ref{Theorem}, the values of the two optimal stopping time problems
(\ref{OptimalStopping1}) and (\ref{OptimalStopping2}) are related by
\begin{equation}\label{relation1}
\lim_{\lambda\uparrow\infty}\mathbf{E}\left[\sup_{t\in[0,T]}|y^{\lambda}_t-y_t|^2\right]=0.
\end{equation}
Moreover, by using the optimal stopping representation (\ref{OptimalStopping2}), we will further
establish the convergence rate of (\ref{relation1}) in Section
\ref{Sec_rate}.



The above optimal stopping with Poisson random intervention times
was firstly introduced by Dupuis and Wang \cite{Dupuis2002} (generalized by Lempa \cite{Lempa} recently), where
they used it to model perpetual American options in a Markovian
setting. Since the state space is one dimensional and the time
horizon is infinite, they did not even need to introduce any
penalized equation. Instead, they worked out two ordinary
differential equations (\emph{ODE} for short) defined in continuity
region and stopping region respectively. Recently, Liang et al
\cite{Liang} established a connection between such kind of optimal
stopping with Poisson random intervention times and dynamic bank run
problems. In a Markovian setting, Dai et al \cite{Dai} intuitively
showed that the penalty method for their optimal stopping time
problem is closely related to some intensity framework. However,
they did not introduce any stochastic control
interpretation for their penalty method.\\

The paper is organized as follows: Theorem \ref{Theorem} is proved
in Section \ref{Sec_Proof}. Then we provide four applications of the
optimal stopping representation (\ref{OptimalStopping2}) in the
following sections. In Section \ref{Sec_rate} we give the
convergence rate of the penalized BSDE (\ref{penalizedBSDE1}) in a
Markovian setting. We also give an optimal control representation
for (\ref{penalizedBSDE1}) in the sense of randomized stopping in
Section \ref{Sec_Control}. Then in Section \ref{Sec_Multi}, we apply
to multidimensional reflected (oblique) BSDE, and give two optimal
switching representations for the associated multidimensional
penalized BSDEs, one of which is closely related to  BSDE with
regime switching. In Section \ref{Sec_Z}, we apply to reflected BSDE
with a convex constraint on $Z$ (\emph{constrained reflected BSDE}
for short), and give an optimal control/optimal stopping
representation for the associated penalized BSDE. Finally, Section 7
concludes.

\section{Proof of Theorem \ref{Theorem}}\label{Sec_Proof}

The optimal stopping time problem (\ref{OptimalStopping2}) has a
constraint on its control set, i.e. the optimal stopping time must
be chosen from the arrival times $\{T_n\}_{n\geq 0}$ of the
underlying Poisson process. Given the arrival time $T_n$, by
defining pre-$T_n$ $\sigma$-field
$$\mathcal{G}_{T_n}=\left\{A\in\bigvee_{s\geq 0}\mathcal{G}_{s}:
A\cap\{T_n\leq s\}\in\mathcal{G}_s\ \text{for}\ s\geq 0\right\}$$
and denoting
$\tilde{\mathbb{G}}=\{\mathcal{G}_{T_n}\}_{n\geq 0}$,
it is obvious that the problem (\ref{OptimalStopping2}) is
equivalent to the following discrete optimal stopping time problem
(where the control constraint does not appear): Conditional on $\{T_{i-1}\leq t<T_i\}$,
\begin{equation}\label{OptimalStopping3}
y_{t}^{\lambda}=\esssup_{N\in\mathcal{N}_{i}(\lambda)}
\mathbf{E}\left[\int_{t}^{T_N\wedge
T}f_s(Y_s^{\lambda},Z_s^{\lambda})ds+S_{T_N}\mathbf{1}_{\{T_N<
T\}}+\xi\mathbf{1}_{\{T_N\geq
T\}}|\mathcal{G}_{t}\right],
\end{equation}
where
$$\mathcal{N}_{i}(\lambda)=\left\{\tilde{\mathbb{G}}\text{-stopping\ time}\ N\
\text{for}\ i\leq N\leq M+1.\right\}$$

Once again, the subscript $i$ in $\mathcal{N}_{i}(\lambda)$
represents the smallest stopping time that is allowed to choose, and
$\lambda$ represents the intensity of the underlying filtration
$\tilde{\mathbb{G}}$. Note that (\ref{OptimalStopping3}) is a
discrete optimal stopping problem, as the player is allowed to stop
at a sequence of integers $i,i+1,\dots,M+1$. The optimal stopping
time is then some integer-valued random variable $N^*_{i}$ such that
$N^*_{i}=\inf\{N\geq i: Y^{\lambda}_{T_N}\leq S_{T_N}\}\wedge
(M+1).$ In the following, we will work on the optimal stopping time
problem with the form (\ref{OptimalStopping3}).



\subsection{Representation for Linear Case}\label{subSec2.1}

In this section, we consider the case where the driver $f_s(y,z)$ is
independent of $(y,z)$, and simply write it as $f_s$ in such a
situation. Note that the corresponding reflected BSDE (\ref{RBSDE1})
becomes linear, and so is the optimal stopping representation
(\ref{OptimalStopping1}).

\begin{lemma}\label{lemma0}
Suppose that Assumption \ref{Assumption} holds, and that
$f_s(y,z)=f_s$. Then conditional on $\{T_{i-1}\leq t<T_{i}\}$, the
solution of the penalized BSDE (\ref{penalizedBSDE1}) is the unique
solution of the following recursive equation
\begin{equation}\label{DPEforBSDE1}
Y_{t}^{\lambda}=\mathbf{E}\left[\int_{t}^{T_{i}\wedge
T}f_sds+\max\left\{S_{T_{i}},Y^{\lambda}_{T_{i}}\right\}
\mathbf{1}_{\{T_{i}\leq T\}}+\xi\mathbf{1}_{\{T_{i}>
T\}}|\mathcal{G}_{t}\right].
\end{equation}
\end{lemma}

\begin{proof}
We introduce the dual equation for the penalized BSDE
(\ref{penalizedBSDE1}),
$$\alpha_t=1-\int_0^{t}\lambda\alpha_sds,\ \text{for}\ t\in[0,T].$$
Applying It\^o's formula to $\alpha_tY_t^{\lambda}$, we obtain
\begin{align*}
\alpha_{t}Y_{t}^{\lambda}=\alpha_TY_T^{\lambda}+\int_{t}^{T}\alpha_s\left(f_s+
\lambda\max\{S_s,Y_s^{\lambda}\}\right)ds-\int_{t}^{T}\alpha_sZ_s^{\lambda}dW_s,
\end{align*}
so that
\begin{align}\label{DPEforBSDE2}
Y_{t}^{\lambda}&=\frac{\alpha_T}{\alpha_{t}}\xi+\int_{t}^{T}\frac{\alpha_{s}}{\alpha_{T_n}}
\left(f_s+
\lambda\max\left\{S_{s},Y_{s}^{\lambda}\right\}\right)ds-\int_{t}^T\frac{\alpha_s}{\alpha_{t}}Z_s^{\lambda}dW_s\\
&=\mathbf{E}\left[e^{-\lambda(T-{t})}\xi+\int_{t}^Te^{-\lambda(s-T_n)}\left(f_s+
\lambda\max\left\{S_{s},Y_{s}^{\lambda}\right\}\right)ds
|\mathcal{F}_{t}\right].\nonumber
\end{align}

Next, conditional on $\{T_{i-1}<t<T_i\}$, we use the conditional density $\lambda e^{-\lambda
(x-t)}dx$ of $T_{i}-t$ to calculate (\ref{DPEforBSDE2}):
\begin{align*}
\mathbf{E}\left[\int_{t}^{T_{i}\wedge
T}f_sds|\mathcal{G}_{t}\right]
=&\ \mathbf{E}\left[\int_{t}^{T_{i}\wedge
T}f_sds|\mathcal{F}_{t}\right]\\
=&\
\mathbf{E}\left[e^{-\lambda(T-t)}\int_{t}^Tf_sds+\int_{t}^T\lambda
e^{-\lambda(x-t)}(\int_{t}^xf_udu)dx|\mathcal{F}_{t}\right]\\
=&\
\mathbf{E}\left[\int_{t}^Te^{-\lambda(s-t)}f_sds|\mathcal{F}_{t}\right],
\end{align*}
where we used integration by parts in the second equality.
Similarly, we have that
\begin{align*}
&\
\mathbf{E}\left[\max\left\{S_{T_{i}},Y_{T_{i}}^{\lambda}\right\}\mathbf{1}_{\{T_{i}\leq
T\}}+\xi\mathbf{1}_{\{T_{i}> T\}}|\mathcal{G}_{t}\right]\\
=&\ \mathbf{E}\left[\int_{t}^T\lambda
e^{-\lambda(s-{t})}\max\left\{S_{s},Y_{s}^{\lambda}\right\}ds
+e^{-\lambda(T-t)}\xi|\mathcal{F}_{t}\right].
\end{align*}
Hence, we obtain (\ref{DPEforBSDE1}) on $\{T_{i-1}<t<T_i\}$ by
plugging the above two expressions into (\ref{DPEforBSDE2}).

It is similar to obtain (\ref{DPEforBSDE1}) on $T_{i-1}$:
\begin{equation*}
Y_{T_{i-1}}^{\lambda}=\mathbf{E}\left[\int_{T_{i-1}}^{T_{i}\wedge
T}f_sds+\max\left\{S_{T_{i}},Y_{T_{i}}^{\lambda}\right\}\mathbf{1}_{\{T_{i}\leq
T\}}+\xi\mathbf{1}_{\{T_{i}> T\}}|\mathcal{G}_{T_{i-1}}\right].
\end{equation*}

Since the recursive equation (\ref{DPEforBSDE1}) obviously admits a
unique solution, $Y^{\lambda}_{t}$ is then the unique solution to
(\ref{DPEforBSDE1}).
\end{proof}

As a direct consequence of Lemma \ref{lemma0}, if we define
$\widehat{Y}^{\lambda}=\max\left\{S,Y^{\lambda}\right\}$, then
$\widehat{Y}^{\lambda}$ satisfies the following recursive equation: For $1\leq i\leq M+1$,
\begin{equation}\label{DPEFORBSDE3}
\widehat{Y}_{T_{i-1}}^{\lambda}=\max\left\{S_{T_{i-1}},
\mathbf{E}\left[\int_{T_{i-1}}^{T_{i}\wedge
T}f_sds+\widehat{Y}^{\lambda}_{T_{i}} \mathbf{1}_{\{T_{i}\leq
T\}}+\xi\mathbf{1}_{\{T_{i}>
T\}}|\mathcal{G}_{T_{i-1}}\right]\right\},
\end{equation}
which admits a unique solution, as we can calculate its solution
backwards in a recursive way.

In the following, we show that $\widehat{Y}^{\lambda}_{T_{i-1}}$ is the value
of another optimal stopping problem. Introduce an auxiliary optimal
stopping problem associated with (\ref{OptimalStopping3}):
\begin{equation}\label{OptimalStopping4}
\widehat{y}_{T_{i-1}}^{\lambda}=\esssup_{N\in{{\mathcal{N}}}_{i-1}(\lambda)}
\mathbf{E}\left[\int_{T_{i-1}}^{T_N\wedge
T}f_sds+S_{T_N}\mathbf{1}_{\{T_N< T\}}+\xi\mathbf{1}_{\{T_N\geq
T\}}|\mathcal{G}_{T_{i-1}}\right],
\end{equation}
where
$${{\mathcal{N}}}_{i-1}(\lambda)=\left\{\tilde{\mathbb{G}}\text{-stopping\ time}\ N\ \text{for}\
i-1\leq N\leq M+1.\right\}$$

The difference between (\ref{OptimalStopping4}) and
(\ref{OptimalStopping3}) starting from $T_{i-1}$ is that the former
is allowed to stop at the initial starting time $T_{i-1}$, while the
latter not.


\begin{lemma}\label{Lemma1} Suppose that Assumption \ref{Assumption} holds, and that
$f_s(y,z)=f_s$. For any integer $1\leq i\leq M+1$, the value
$\widehat{y}^{\lambda}_{T_{i-1}}$ of the auxiliary optimal stopping
time problem (\ref{OptimalStopping4}) satisfies the recursive
equation (\ref{DPEFORBSDE3}):
\begin{equation*}
\widehat{y}^{\lambda}_{T_{i-1}}=\max\left\{S_{T_{i-1}},\mathbf{E}
\left[\int_{T_{i-1}}^{T_{i}\wedge
T}f_sds+\widehat{y}^{\lambda}_{T_{i}}\mathbf{1}_{\{T_{i}< T\}}
+\xi\mathbf{1}_{\{T_{i}\geq
T\}}|\mathcal{G}_{T_{i-1}}\right] \right\}.
\end{equation*}
The optimal stopping time is given by
$\widehat{N}^*_{i-1}=\inf\{N\geq i-1:
\widehat{y}^{\lambda}_{T_N}\leq S_{T_N}\}\wedge (M+1).$ Hence,
$\widehat{Y}^{\lambda}_{T_{i-1}}=\widehat{y}^{\lambda}_{T_{i-1}}$
$a.s.$.
\end{lemma}

\begin{proof} Define the following processes
\begin{align*}
\bar{y}^{\lambda}_{t}&=\widehat{y}^{\lambda}_{t}+\int_0^{t}f_sds;\\
\bar{S}_{t}&=S_{t}+\int_0^{t}f_sds;\\
\bar{\xi}&=\xi+\int_0^{T}f_sds.
\end{align*}
Since $T_{M}\leq T< T_{M+1}$,
the auxiliary optimal stopping problem (\ref{OptimalStopping4})
is equivalent to
\begin{align*}
\bar{y}_{T_{i-1}}^{\lambda}&=\esssup_{N\in{\mathcal{N}}_{i-1}(\lambda)}
\mathbf{E}\left[\bar{S}_{T_N}\mathbf{1}_{\{T_N<
T\}}+\bar{\xi}\mathbf{1}_{\{T_N\geq T\}}|\mathcal{G}_{T_{i-1}}\right]\\
&=\esssup_{N\in{\mathcal{N}}_{i-1}(\lambda)}
\mathbf{E}\left[\bar{S}_{T_N}\mathbf{1}_{\{i-1\leq N\leq
M\}}+\bar{\xi}\mathbf{1}_{\{N=M+1\}}|\mathcal{G}_{T_{i-1}}\right].
\end{align*}

We claim that
\begin{equation}\label{DPE2}
\left\{
\begin{array}{l}
\bar{y}_{T_M}^{\lambda}=\max\left\{\bar{S}_{T_M},
\mathbf{E}\left[\bar{\xi}|\mathcal{G}_{T_M}\right]\right\},\\[+0.2cm]
\bar{y}_{T_n}^{\lambda}=\max\left\{\bar{S}_{T_n},
\mathbf{E}\left[\bar{y}^{\lambda}_{T_{n+1}}|\mathcal{G}_{T_n}\right]\right\},
\ \ \ \text{for}\ i-1\leq n\leq M-1.
\end{array}%
\right.
\end{equation}%
If (\ref{DPE2}) holds, then
\begin{align*}
\bar{y}_{T_{i-1}}^{\lambda}&=\max\left\{\bar{S}_{T_{i-1}},
\mathbf{E}\left[\bar{y}_{T_{i}}^{\lambda}\mathbf{1}_{\{i\leq
M\}}+\bar{\xi}\mathbf{1}_{\{i>
M\}}|\mathcal{G}_{T_{i-1}}\right]\right\}\\
&=\max\left\{\bar{S}_{T_{i-1}},
\mathbf{E}\left[\bar{y}_{T_{i}}^{\lambda} \mathbf{1}_{\{T_{i}<
T\}}+\bar{\xi}\mathbf{1}_{\{T_{i}\geq
T\}}|\mathcal{G}_{T_{i-1}}\right]\right\},
\end{align*}
which is the recursive equation (\ref{DPEFORBSDE3}) if we express
the above equation in terms of $\widehat{y}^{\lambda}$, $S$ and
$\xi$.

Therefore, in order to complete the proof, we only need to show
(\ref{DPE2}). Indeed, for $n=M$,
\begin{align*}
\bar{y}_{T_M}^{\lambda}&=\esssup_{N\in{\mathcal{N}}_M(\lambda)}
\mathbf{E}\left[\bar{S}_{T_N}\mathbf{1}_{\{N= M\}}+
\bar{\xi}\mathbf{1}_{\{N=M+1\}}|\mathcal{G}_{T_M}\right]\\
&=\max\left\{\bar{S}_{T_M},\mathbf{E}\left[\bar{\xi}|\mathcal{G}_{T_M}\right]\right\}.
\end{align*}
In general, for $i-1\leq n\leq M-1$,
\begin{align*}
\bar{y}_{T_n}^{\lambda} &=\esssup_{N\in{\mathcal{N}}_n(\lambda)}
\mathbf{E}\left[\bar{S}_{T_N}\mathbf{1}_{\{n\leq N\leq
M\}}+\bar{\xi}\mathbf{1}_{\{N=M+1\}}|\mathcal{G}_{T_n}\right]\\
&=\esssup_{N\in{\mathcal{N}}_n(\lambda)}
\mathbf{E}\left[\mathbf{E}\left[\bar{S}_{T_N}\mathbf{1}_{\{n\leq
N\leq
M\}}+\bar{\xi}\mathbf{1}_{\{N=M+1\}}|\mathcal{G}_{T_{n+1}}\right]|\mathcal{G}_{T_n}\right]\\
&=\esssup_{N\in{\mathcal{N}}_n(\lambda)}\mathbf{E}\left[\bar{S}_{T_N}\mathbf{1}_{\{N=n\}}+
\mathbf{E}\left[\bar{S}_{T_N}\mathbf{1}_{\{n+1\leq N\leq
M\}}+\bar{\xi}\mathbf{1}_{\{N=M+1\}}|\mathcal{G}_{T_{n+1}}\right]|\mathcal{G}_{T_n}\right]\\
&=\max\left\{\bar{S}_{T_n},\mathbf{E}\left[\bar{y}^{\lambda}_{T_{n+1}}|\mathcal{G}_{T_n}\right]\right\}.
\end{align*}

Finally, we prove that $\widehat{N}^*_{i-1}$ is indeed the optimal
stopping time for the auxiliary optimal stopping problem
(\ref{OptimalStopping4}). For this, it suffices to show that
$\widehat{y}^{\lambda}_{T_{m\wedge \widehat{N}^*_{i-1}}}$ for $m\geq
i-1$ is a $\tilde{\mathbb{G}}$-martingale:
\begin{align*}
\mathbf{E}\left[\widehat{y}^{\lambda}_{T_{(m+1)\wedge\widehat{N}^*_{i-1}}}|\mathcal{G}_{T_m}\right]&
=\mathbf{E}\left[\left(\sum_{j=i-1}^m\mathbf{1}_{\{\widehat{N}_{i-1}^*=j\}}+\mathbf{1}_{\{\widehat{N}_{i-1}^*\geq
m+1\}}\right)\widehat{y}^{\lambda}_{T_{(m+1)\wedge\widehat{N}^*_{i-1}}}|\mathcal{G}_{T_m}\right]\\
&=\mathbf{E}\left[\sum_{j={i-1}}^m\mathbf{1}_{\{\widehat{N}_{i-1}^*=j\}}\widehat{y}_{T_{j}}^{\lambda}
+\mathbf{1}_{\{\widehat{N}_{i-1}^*>m\}}\widehat{y}^{\lambda}_{T_{m+1}}|\mathcal{G}_{T_m}\right]\\
&=\sum_{j=i-1}^m\mathbf{1}_{\{\widehat{N}_{i-1}^*=j\}}\widehat{y}_{T_{j}}^{\lambda}+
\mathbf{1}_{\{\widehat{N}_{i-1}^*>m\}}\mathbf{E}\left[\widehat{y}^{\lambda}_{T_{m+1}}|\mathcal{G}_{T_{m}}\right]\\
&=\sum_{j=i-1}^m\mathbf{1}_{\{\widehat{N}_{i-1}^*=j\}}\widehat{y}_{T_{j}}^{\lambda}+
\mathbf{1}_{\{\widehat{N}_{i-1}^*>m\}}\widehat{y}^{\lambda}_{T_m}=\widehat{y}^{\lambda}_{T_{m\wedge
\widehat{N}_{i-1}^*}},
\end{align*}
where we used the definition of $\widehat{N}^*_{i-1}$ is the second
last equality, and the proof is complete.
\end{proof}

We are now in a position to prove the linear situation of Theorem
\ref{Theorem}. From Lemma \ref{lemma0} and the definition of
$\widehat{Y}^{\lambda}$, conditional on $\{T_{i-1}\leq t< T_i\}$,
$$Y_{t}^{\lambda}=\mathbf{E}\left[\int_{t}^{T_{i}\wedge
T}f_sds+\widehat{Y}^{\lambda}_{T_{i}}\mathbf{1}_{\{T_{i}<
T\}}+\xi\mathbf{1}_{\{T_{i}\geq T\}}|\mathcal{G}_{t}\right].$$
Thanks to Lemma \ref{Lemma1},
$\widehat{Y}^{\lambda}_{T_{i}}=\widehat{y}^{\lambda}_{T_{i}}$, which
is the value of the auxiliary optimal stopping problem
(\ref{OptimalStopping4}) starting from $T_{i}$. Hence, for any
$\tilde{\mathbb{G}}$-stopping time $N\in{\mathcal{N}}_{i}(\lambda)$,
\begin{align*}
Y_{t}^{\lambda}&=\mathbf{E}\left[\int_{t}^{T_{i}\wedge
T}f_sds+\widehat{y}^{\lambda}_{T_{i}}\mathbf{1}_{\{T_{i}<
T\}}+\xi\mathbf{1}_{\{T_{i}\geq T\}}|\mathcal{G}_{t}\right]\\
&\geq \mathbf{E}\left[\int_{t}^{T_{i}\wedge T}f_sds+
\mathbf{E}\left[\int_{T_{i}}^{T_N\wedge
T}f_sds+S_{T_{N}}\mathbf{1}_{\{T_N< T\}}+\xi\mathbf{1}_{\{T_{N}\geq
T\}}|\mathcal{G}_{T_{i}}\right]
\mathbf{1}_{\{T_{i}<T\}}\right.\\
&\ \ \ \ \ \ \left.+\ \xi\mathbf{1}_{\{T_{i}\geq T\}}|\mathcal{G}_{t}\right]\\[+0.2cm]
&=\mathbf{E}\left[\int_{t}^{T_{i}\wedge
T}f_sds+\left(\int_{T_{i}}^{T_N\wedge
T}f_sds\right)\mathbf{1}_{\{T_{i}<T\}}+S_{T_N}\mathbf{1}_{\{T_N<T,T_{i}<T\}}\right.\\
&\ \ \ \ \ \ \left.+\ \xi\left(\mathbf{1}_{\{T_{N}\geq T,\
T_{i}<T\}}+\mathbf{1}_{\{T_{i}\geq
T\}}\right) |\mathcal{G}_{t}\right]\\
&=\mathbf{E}\left[\int_{t}^{T_N\wedge
T}f_sds+S_{T_N}\mathbf{1}_{\{T_N< T\}}+\xi\mathbf{1}_{\{T_N\geq
T\}}|\mathcal{G}_{t}\right],
\end{align*}
where we used the fact that $\{T_{i}\geq T\}\subset\{T_{N}\geq T\}$
in the last two equalities. By taking the supremum over
$N\in\mathcal{N}_{i}(\lambda)$, we obtain that $Y_{t}^{\lambda}\geq
y_{t}^{\lambda}$.

We now choose $N=\widehat{N}^*_{i}$, where $\widehat{N}^*_{i}$ is
the optimal stopping time for $\widehat{y}_{T_{i}}^{\lambda}$ given
in Lemma \ref{Lemma1}, to get the reverse inequality. Indeed,
\begin{align*}
Y_{t}^{\lambda}&=\mathbf{E}\left[\int_{t}^{T_{i}\wedge
T}f_sds+\widehat{y}^{\lambda}_{T_{i}}\mathbf{1}_{\{T_{i}<
T\}}+\xi\mathbf{1}_{\{T_{i}\geq T\}}|\mathcal{G}_{t}\right]\\
&=\mathbf{E}\left[\int_{{t}}^{T_{i}\wedge T}f_sds\right.+
\mathbf{E}\left[\int_{T_{i}}^{T_{\widehat{N}^*_{i}}\wedge
T}f_sds+S_{T_{\widehat{N}^*_{i}}}\mathbf{1}_{\{T_{\widehat{N}^*_{i}}<
T\}}+\ \xi\mathbf{1}_{\{T_{\widehat{N}^*_{i}}\geq
T\}}|\mathcal{G}_{T_{i}}\right] \mathbf{1}_{\{T_{i}<
T\}}\\
&\ \ \ \ \ \ \left.+\ \xi\mathbf{1}_{\{T_{i}\geq
T\}}|\mathcal{G}_{t}\right]\\[+0.2cm]
&=\mathbf{E}\left[\int_{t}^{T_{i}\wedge
T}f_sds+\left(\int_{T_{i}}^{T_{\widehat{N}^*_{i}}\wedge
T}f_sds\right)\mathbf{1}_{\{T_{i}<T\}}+S_{T_{\widehat{N}^*_{i}}}\mathbf{1}_
{\{T_{\widehat{N}^*_{i}}<T,T_{i}<T\}}\right.\\
&\ \ \ \ \ \ \left.+\
\xi\left(\mathbf{1}_{\{T_{\widehat{N}^*_{i}}\geq T,\
T_{i}<T\}}+\mathbf{1}_{\{T_{i}\geq
T\}}\right) |\mathcal{G}_{t}\right]\\
&=\mathbf{E}\left[\int_{t}^{T_{\widehat{N}_{i}^*}\wedge
T}f_sds+S_{T_{\widehat{N}_{i}^*}}\mathbf{1}_{\{T_{\widehat{N}_{i}^*}<
T\}}+\xi\mathbf{1}_{\{T_{\widehat{N}^*_{i}}\geq
T\}}|\mathcal{G}_{t}\right]\leq y_{t}^{\lambda}.
\end{align*}
Hence, $Y_{t}^{\lambda}=y_{t}^{\lambda}$, and the optimal stopping
time is $\widehat{N}^*_{i}$, which is just $N^*_{i}$ defined at the
beginning of Section \ref{Sec_Proof},
\begin{align*}
\widehat{N}_{i}^*&=\inf\{N\geq i: \widehat{y}^{\lambda}_{T_N}\leq
S_{T_N}\}\wedge
(M+1)\\
&=\inf\{N\geq i: \widehat{Y}^{\lambda}_{T_N}\leq S_{T_N}\}\wedge
(M+1)\\
&=\inf\{N\geq i: Y^{\lambda}_{T_N}\leq S_{T_N}\}\wedge
(M+1)=N^*_{i}.
\end{align*}

\subsection{Representation for Nonlinear Case}\label{sec2.2}

In this section, we extend the optimal stopping representation to the
nonlinear case, and complete the proof of Theorem \ref{Theorem}.

Denote $(Y^{\lambda},Z^{\lambda})$ as the unique solution to the
penalized BSDE (\ref{penalizedBSDE1}). Consider the optimal stopping
time problem (\ref{OptimalStopping3}) conditional on $\{T_{i-1}\leq
t<T_i\}$:
\begin{equation*}
y_t^{\lambda}=\esssup_{N\in\mathcal{N}_{i}(\lambda)}
\mathbf{E}\left[\int_t^{T_N\wedge
T}f_s(Y_s^{\lambda},Z_s^{\lambda})ds+S_{T_N}\mathbf{1}_{\{T_N<
T\}}+\xi\mathbf{1}_{\{T_N\geq T\}}|\mathcal{G}_t\right].
\end{equation*}
From Section \ref{subSec2.1}, $y_t^{\lambda}=\tilde{Y}^{\lambda}_t$
admits the following BSDE representation
\begin{equation*}
\tilde{Y}^{\lambda}_t=\xi+\int_t^{T}f_s(Y_s^{\lambda},Z_s^{\lambda})ds
+\int_t^T\lambda\max\{0,S_s-\tilde{Y}^{\lambda}_s\}ds-\int_t^{T}\tilde{Z}^{\lambda}_tdW_s.
\end{equation*}
On the other hand, $(Y^{\lambda},Z^{\lambda})$ satisfies the
penalized BSDE (\ref{penalizedBSDE1})
\begin{equation*}
Y_t^{\lambda}=\xi+\int_t^{T}f_s(Y_s^{\lambda},Z_s^{\lambda})ds
+\int_t^T\lambda\max\{0,S_s-Y^{\lambda}_s\}ds-\int_t^{T}Z_s^{\lambda}dW_s.
\end{equation*}
Define
$$\delta Y_t^{\lambda}=\tilde{Y}^{\lambda}_t-Y_t^{\lambda};\
\delta Z_t^{\lambda}=\tilde{Z}^{\lambda}_t-Z_t^{\lambda}.$$ Then
$(\delta Y^{\lambda},\delta Z^{\lambda})$ satisfies the following
linear BSDE
\begin{equation}\label{differenceBSDE}
\delta Y_t^{\lambda}=\int_t^{T}\lambda\beta_s\delta
Y_s^{\lambda}ds-\int_t^T\delta Z^{\lambda}_sdW_s
\end{equation}
with
$$\beta_s=\frac{\max\{0,S_s-\tilde{Y}_s^{\lambda}\}-\max\{0,S_s-Y_s^{\lambda}\}}
{\delta Y_s^{\lambda}}\times\mathbf{1}_{\{\delta Y_s^{\lambda}\neq
0\}}.$$ Obviously, $|\beta_s|\leq 1$, so BSDE (\ref{differenceBSDE})
admits a unique solution (see for example \cite{ElKaroui19973} for
the proof). On the other hand, $\delta Y_t^{\lambda}=\delta
Z_t^{\lambda}=0$ is one obvious solution to BSDE
(\ref{differenceBSDE}). Therefore, we conclude that
$\tilde{Y}^{\lambda}_t=Y_t^{\lambda}$ $a.s.$, which
proves Theorem \ref{Theorem}.

We conclude this section by reformulating the optimal stopping representation
(\ref{OptimalStopping2}) as the following remark, which will be used in Section
\ref{Sec_rate}.


\begin{remark}\label{remark}
Suppose that Assumption \ref{Assumption} holds. Then for any integer
$i\geq 1$, conditional on $\{T_{i-1}\wedge T\leq t<T_{i}\wedge T\}$,
the solution to the penalized BSDE (\ref{penalizedBSDE1}) is the
value of the optimal stopping time (\ref{OptimalStopping2}):
$Y_t^{\lambda}=y_t^{\lambda}$ $a.s.$. Moreover, the value
$y_t^{\lambda}$ satisfies the recursive equation:
\begin{align*}\label{DPEforBSDE11}
y_{t}^{\lambda}&=\mathbf{E}\left[\int_{t}^{T_{i}\wedge
T}f_s(Y_s^{\lambda},Z_s^{\lambda})ds+\max\left\{S_{T_{i}},y^{\lambda}_{T_{i}}\right\}
\mathbf{1}_{\{T_{i}\leq T\}}+\xi\mathbf{1}_{\{T_{i}>
T\}}|\mathcal{G}_{t}\right]\notag\\
&=\mathbf{E}\left[\int_{t}^{T_{i}\wedge
T}f_s(Y_s^{\lambda},Z_s^{\lambda})ds+\max\left\{S_{T_{i}},y^{\lambda}_{T_{i}}\right\}
\mathbf{1}_{\{T_{i}\leq T\}}+\xi\mathbf{1}_{\{T_{i}>
T\}}|\mathcal{F}_{t}\right].
\end{align*}
\end{remark}


\section{Application I: Convergence Rate of Penalized BSDE}
\label{Sec_rate}

The penalization method only provides the convergence of the
solution $(Y^{\lambda},Z^{\lambda},K^{\lambda})$ of the penalized
BSDE (\ref{penalizedBSDE1}) to the solution $(Y,Z,K)$ of the
reflected BSDE (\ref{RBSDE1}), but without any convergence rate,
because the proof of the convergence is based on compactness
arguments. What is even worse is that the penalized BSDE
(\ref{penalizedBSDE1}) does not provide an efficient numerical
algorithm, as the Lipschitz constant of the driver depends on
$\lambda$ which will explode when $\lambda\uparrow\infty$. Actually,
it is still an open question on how to numerically approximate the
corresponding penalized BSDE (\ref{penalizedBSDE1}) with an even
fixed (but large) intensity $\lambda$ (see Page 26 in
\cite{Chassagneux1}).

Thanks to our optimal stopping representation, the penalized BSDE
(\ref{penalizedBSDE1}) is nothing but a random time discretization
of the optimal stopping representation for the corresponding
reflected BSDE (\ref{RBSDE1}), where the time is discretized by
Poisson arrival times. On the other hand, it has been known the
convergence rate of the fixed time discretization of the optimal
stopping representation for (\ref{RBSDE1}), so called the Bermudan
approximation in \cite{Chassagneux0} and
\cite{Ma}. Hence, it is plausible to obtain the convergence rate of
the penalized BSDE (\ref{penalizedBSDE1}), or equivalently, the
convergence rate of the optimal stopping representation
(\ref{OptimalStopping2}).

\begin{assumption}\label{Assumption1}
\begin{itemize}
\item The terminal data $\xi$, the driver $f_s(y,z)$ and
the obstacle $S$ satisfy Assumption \ref{Assumption};
\item Moreover, the driver
$f_s(y,z)=f(X_s,y,z)$, the terminal date $\xi=g(X_T)$ for $g(\cdot)$
being Lipschitz continuous, and the obstacle process $S_s=l(X_s)$
for $l(\cdot)\in C^2$, where $X$ is a diffusion process with enough
regality.
\end{itemize}
\end{assumption}

We refer to \cite{Chassagneux0,Chassagneux1,Ma} for more detail
assumptions on the diffusion $X$. In the following, we improve the
convergence (\ref{relation1}) by giving its convergence rate.

\begin{proposition} Suppose that Assumption
\ref{Assumption1} holds. Then for any integer $M\geq 1$, the value
of the optimal stopping time problem (\ref{OptimalStopping2}) will
converge to the value of (\ref{OptimalStopping1}) with the following
rate:
\begin{equation}\label{relation2}
\mathbf{E}\left[\sup_{t\in[0,T_M\wedge
T]}\mathbf{E}\left[|y_t^{\lambda}-y_t|^2\right]\right]\leq
\frac{C}{\lambda}.
\end{equation}
for some constant $C$.
\end{proposition}

\begin{proof} For any $M\geq 1$,
Theorem \ref{Theorem} and Remark \ref{remark} imply that
$$y_t^{\lambda}=\mathbf{E}\left[\int_{t}^{T_{i}\wedge
T}f_s(Y_s^{\lambda},Z_s^{\lambda})ds+
\hat{y}^{\lambda}_{T_i}\mathbf{1}_{\{T_{i}\leq
T\}}+g(X_T)\mathbf{1}_{\{T_{i}> T\}}|\mathcal{F}_{t}\right]$$
conditional on $t\in[T_{i-1}\wedge T,T_i\wedge T)$, where
$\hat{y}^{\lambda}_{T_i}=\max\left\{l(X_{T_{i}}),y^{\lambda}_{T_{i}}\right\}$
for $1\leq i\leq M$. This is exactly the Bermudan approximation of
the optimal stopping time problem (\ref{OptimalStopping1}) if we
condition on $\vee_{t\geq 0}\mathcal{H}_t$. Hence, by a similar
argument as in Proposition 3.1 of \cite{Chassagneux1} (see also
Section 4 of \cite{Chassagneux0} and Section 3 of \cite{Ma}),
conditional on $\vee_{t\geq 0}\mathcal{H}_t$, we obtain that
\begin{equation}\label{inter_step}
\sup_{t\in[0,T_M\wedge
T]}\mathbf{E}\left[|y_t^{\lambda}-y_t|^2\right]\leq\max_{1\leq i\leq
M}(T_{i}\wedge T-T_{i-1}\wedge T)\leq\max_{1\leq i\leq
M}(T_{i}-T_{i-1}) ,
\end{equation}
and moreover,
\begin{align*}
\mathbf{E}\left[\sup_{t\in[0,T_M\wedge
T]}\mathbf{E}\left[|y_t^{\lambda}-y_t|^2\right]\right]&=\mathbf{E}\left[\mathbf{E}\left[\left.\sup_{t\in[0,T_M\wedge
T]}\mathbf{E}\left[|y_t^{\lambda}-y_t|^2\right]\right|\vee_{t\geq
0}\mathcal{H}_t\right]\right]\\
&\leq \mathbf{E}\left[\mathbf{E}\left[\left.\max_{1\leq i\leq
M}(T_{i}-T_{i-1})\right|\vee_{t\geq 0}\mathcal{H}_t\right]\right]\\
&=\mathbf{E}\left[\max_{1\leq i\leq M}(T_{i}-T_{i-1})\right].
\end{align*}
The conclusion then follows by observing that $(T_{i-1}-T_i)$ is
exponentially distributed with parameter $\lambda$ and that
\begin{align*}
\mathbf{E}\left[\max_{1\leq i\leq M}(T_{i}-T_{i-1})\right]&=
\int_0^{\infty}\mathbf{P}(\max_{1\leq i\leq M}(T_{i}-T_{i-1})>x)dx\\
&=\int_0^{\infty}(1-(1-e^{-\lambda x})^M)dx\\
&=\frac{1}{\lambda}\int_0^1\frac{1-u^M}{1-u}du=\frac{1}{\lambda}\sum_{i=1}^M\left(\frac{1}{i}\right).
\end{align*}
\end{proof}

\begin{remark}
Thanks to the optimal stopping representation (\ref{OptimalStopping2}), it is also possible
to obtain a numerical algorithm to solve the penalized BSDE
(\ref{penalizedBSDE1}), where the parameter $\lambda$ is hidden in
the Poisson arrival times $\{T_i\}_{i\geq 1}$, and we only need to
numerically solve the BSDE with the standard driver $f(x,y,z)$
instead of $f(x,y,z)+\lambda\max\{0,l(x)-y\}$:
\begin{align*}
Y_{t}^{\lambda}=&\
\max\left\{l(X_{T_{i}}),Y^{\lambda}_{T_{i}}\right\}
\mathbf{1}_{\{T_{i}\leq T\}}+g(X_T)\mathbf{1}_{\{T_{i}> T\}}\\
&\ + \int_{t}^{T_{i}\wedge
T}f(X_s,Y_s^{\lambda},Z_s^{\lambda})ds-\int_t^{T_i\wedge
T}Z_s^{\lambda}dW_s
\end{align*}
on $\{T_{i-1}\leq t<T_i\}$. Since the numerical approximation is of
independent interest, we will leave it for future research.
\end{remark}

\section{Application II: Randomized Stopping and Optimal Control Representation}
\label{Sec_Control}

Krylov in \cite{Krylov} showed that optimal stopping for controlled
diffusion processes can always be transformed to optimal control by
using randomized stopping. See also Gy\"ongy and Siska \cite{Gyongy}
for its recent development. In this section, our aim is to give
optimal control interpretations of both the reflected BSDE
(\ref{RBSDE1}) and the penalized BSDE (\ref{penalizedBSDE1}).

Let us first recall the basic idea of Krylov's randomized stopping.
For simplicity, we only consider the linear case $f_s(y,z)=f_s$. For
any fixed time $t\in[0,T]$, consider a nonnegative control process
$(r_s)_{s\geq t}$. Let the payoff functional
$\int_t^{\cdot}f_sds+S_{\cdot}$ stop with intensity $r_s\Delta$ in
an infinitesimal interval $(s,s+\Delta)$. Then the probability that
stopping does not occur before time $s$ is
$$e^{-\int_t^{s}r_udu}.$$ The probability that stopping does not occur
before time $s$ and does occur in the infinitesimal interval
$(s,s+{\Delta})$ is $$e^{-\int_t^{s}r_udu}r_s\Delta.$$ Therefore,
the payoff functional associated with the control process $r$ from
$[t,T]$ is given by
\begin{equation*}
\int_t^{T}(\int_t^{s}f_udu+S_s)e^{-\int_t^{s}r_udu}r_sds+(\int_t^{T}f_udu+\xi)
e^{-\int_t^Tr_udu},
\end{equation*}
where the first term is the payoff if stopping does occur before
time $T$, and the second term corresponds to the payoff if stopping
does not occur in the time interval $[t,T]$. By applying integration
by parts, the payoff functional is further simplified to
$$\int_t^{T}(f_s+r_sS_s)e^{-\int_t^{s}r_udu}+e^{-\int_t^{T}r_udu}\xi.$$

We have the following optimal control representation for the
penalized BSDE (\ref{penalizedBSDE1}):

\begin{proposition}
Suppose that Assumption \ref{Assumption} holds. Denote
$(Y^{\lambda},Z^{\lambda})$ as the unique solution to the penalized
BSDE (\ref{penalizedBSDE1}). For any fixed time $t\in[0,T]$, define
the control set $\mathcal{A}(t,\lambda)$ as
$$\mathcal{A}(t,\lambda)=
\left\{\mathbb{F}\text{-adapted process\ }(r_s)_{s\geq t}:\ r_s=0\
\text{or}\ \lambda\right\}.$$ Then the value of the following
optimal control problem
\begin{equation}\label{optimalcontrol}
y_t^{\lambda}=\esssup_{r\in\mathcal{A}(t,\lambda)}\mathbf{E}
\left[\int_t^{T}(f_s(Y_s^{\lambda},Z_s^{\lambda})+r_sS_s)e^{-\int_t^{s}r_udu}ds+e^{-\int_t^{T}r_udu}\xi|\mathcal{F}_t\right]
\end{equation}
is given by the solution to the penalized BSDE
(\ref{penalizedBSDE1}): $y_t^{\lambda}=Y_t^{\lambda}$ $a.s.$ for
$t\in[0,T]$. The optimal control is given by
$r_s^{*}=\lambda\mathbf{1}_{\{Y_s^{\lambda}\leq S_s\}}$ for $s\geq t$.
\end{proposition}

\begin{proof}
We only consider the linear case $f_s(y,z)=f_s$. The proof for the
nonlinear case $f_s(y,z)$ is the same as the one in Section
\ref{sec2.2}.

First, similar to Lemma \ref{lemma0}, it is easy to show that the
following expected payoff process associated with any given control
$r\in\mathcal{A}(t,\lambda)$:
$$y^{\lambda}_t(r)=\mathbf{E}
\left[\int_t^{T}(f_s+r_sS_s)e^{-\int_t^{s}r_udu}+e^{-\int_t^{T}r_udu}\xi|\mathcal{F}_t\right]
$$ is the unique solution to the following linear BSDE
$$y_t^{\lambda}(r)=\xi+\int_t^T\left\{f_s+r_s(S_s-y_s^{\lambda}(r))\right\}ds-\int_t^{T}z_s^{\lambda}(r)dW_s.$$

Note that the control $r$ only appears in the driver. For any
control $r\in\mathcal{A}(t,\lambda)$, we have
\begin{align*}
f_s+r_s(S_s-y_s^{\lambda}(r))&\leq
f_s+\lambda\max\{0,S_s-y_s^{\lambda}(r)\},
\end{align*}
and for $r_s=\lambda\mathbf{1}_{\{y_s^{\lambda}(r)\leq S_s\}}$, we
obtain the equality
$$f_s+\lambda\mathbf{1}_{\{y_s^{\lambda}(r)\leq S_s\}}(S_s-y_s^{\lambda}(r))=f_s+\lambda\max\{0,S_s-y_s^{\lambda}(r)\}.$$

By the BSDE comparison theorem (see for example
\cite{ElKaroui19973}), $y_t^{\lambda}(r)\leq Y_t^{\lambda}$ for any
$r\in\mathcal{A}(t,\lambda)$, where $Y^{\lambda}$ is the solution to
the penalized BSDE (\ref{penalizedBSDE1}):
$$Y_t^{\lambda}=\xi+\int_t^T\left\{f_s+\lambda\max\{0,S_s-Y^{\lambda}_s\}\right\}ds-\int_t^{T}Z^{\lambda}_sdW_s.$$
and $y_t^{\lambda}(r^*)= Y_t^{\lambda}$ for
$r^*_s=\lambda\mathbf{1}_{\{Y_s^{\lambda}\leq S_s\}}$. Since
$y_t^{\lambda}=\esssup_{r\in\mathcal{A}(t,\lambda)}y_t^{\lambda}(r)$,
we conclude that $y_t^{\lambda}=Y_t^{\lambda}$ $a.s.$ for
$t\in[0,T]$, and the optimal control is $r^*_s$ for $s\geq t$.
\end{proof}

\begin{remark}
The optimal control representation for the reflected BSDE
(\ref{RBSDE1}) is the same as (\ref{optimalcontrol}) except that the
control set is changed to
$\mathcal{A}(t)=\cup_{\lambda}\mathcal{A}(t,\lambda)$. As shown in
Krylov \cite{Krylov} for the diffusion case, the value of the
following optimal control problem
\begin{equation}\label{optimalcontrol2}
y_t=\esssup_{r\in\mathcal{A}(t)}\mathbf{E}
\left[\int_t^{T}(f_s(Y_s^{\lambda},Z_s^{\lambda})+r_sS_s)e^{-\int_t^{s}r_udu}ds+e^{-\int_t^{T}r_udu}\xi|\mathcal{F}_t\right]
\end{equation}
is given by the solution to the reflected BSDE (\ref{RBSDE1}):
$y_t=Y_t$ $a.s.$ for $t\in[0,T]$.
\end{remark}
\section{Application III: Multidimensional Reflected BSDE and Regime Switching}
\label{Sec_Multi}

Multidimensional reflected BSDE was firstly introduced by Hamad\`ene
and Jeanblanc \cite{Hamadene1}, where they used its solution to
characterize the value of an optimal switching problem, in
particular in the setting of power plant management. The related
equation was solved by Hu and Tang \cite{Hu} using the penalty
method, and by Hamad\`ene and Zhang \cite{Hamadene2} using the
iterated optimal stopping time method. See also Chassagneux et al
\cite{ELIE} for its recent development. A multidimensional reflected
BSDE is a $d$-dimensional system, where each component $1\leq i\leq
d$ representing regime $i$,
\begin{equation}\label{MRBSDE1}
Y_t^i=\xi^i+\int_t^{T}f_s^i(Y_s,Z_s)ds+\int_t^{T}dK_s^i-\int_t^{T}Z_s^idW_s
\end{equation}
under the constraints
\begin{align*}
\text{(Dominating Condition)}:\ \ \ & Y_t^i\geq \mathcal{M}Y_t^{i}\ \text{for}\ t\in[0,T],\\
\text{(Skorohod Condition)}:\ \ \ &
\int_0^T(Y_t^i-\mathcal{M}Y_t^{i})dK_t^i=0\ \text{for}\ K^i\
\text{continuous\ and\ increasing},
\end{align*}
where the impulse term $\mathcal{M}Y_t^i$ is given by
$$\mathcal{M}Y_t^{i}=\max_{j\neq i}\{Y_t^{j}-C_t^{i,j}\}$$
representing the payoff of switching to regime $j$ from regime $i$.
The terminal data $\xi^i$, the driver $f^i_s(y,z)$ and the switching
cost $(C^{i,j}_s)_{0\leq s\leq T}$ are the given data. Different
from one-dimensional reflected BSDE whose solution must stay above
an obstacle process, the solution of the multidimensional reflected
BSDE (\ref{MRBSDE1}) evolves in the random closed and convex set
$$\left\{y\in\mathbb{R}^d: y^i\geq \max_{j\neq i}\{y^j-C_t^{i,j}\}\right\}.$$

The following standard assumption on the data set
$(\xi^i,f^i,C^{i,j})$ is imposed.

\begin{assumption}\label{Assumption2}
\begin{itemize}
\item The terminal data $\xi^i$ and the driver $f_s^i(y,z)$ satisfy
Assumption \ref{Assumption};
\item The switching cost $(C^{ij})_{1\leq i,j\leq d}$ is a bounded
$\mathbb{F}$-adapted process satisfying (i) $C^{ii}_t=0$; (ii)
$\inf_{t\in[0,T]}C_t^{ij}\geq C>0$ for $i\neq j$; and (iii)
$\inf_{t\in[0,T]}C_t^{ij}+C_t^{jl}-C_t^{il}\geq C>0$ for $i\neq
j\neq l$.
\end{itemize}
\end{assumption}

In Hu and Tang \cite{Hu}, they further assume that
$f^i_s(y,z)=f^i_s(y^i,z^i)$ so that (\ref{MRBSDE1}) admits a
solution. This condition was relaxed in Hamad\`ene and Zhang
\cite{Hamadene2} and Chassagneux et al \cite{ELIE}, where the driver
is even allowed to be coupled in $y$, i.e. having the form
$f_s^i(y,z^i)$. However, it is still an open problem for the case of
the fully coupled driver $f^i_s(y,z)$.

Under Assumption \ref{Assumption2} with the decoupled driver
$f^i_s(y,z)=f^i_s(y^i,z^i)$, Hu and Tang \cite{Hu} proved that the
solution to the multidimensional reflected BSDE (\ref{MRBSDE1})
corresponds to the value of an optimal switching problem. Indeed,
introduce the control set $\mathcal{K}_i(t)$ as
$$\mathcal{K}_{i}(t)=
\left\{\mathbb{F}\text{-adapted process\ }(u_s)_{s\geq t}:\
u_s=\alpha_0\mathbf{1}_{[t,\tau_1]}(s)+\sum_{k\geq
1}\alpha_{k}\mathbf{1}_{(\tau_{k},\tau_{k+1}]}(s)\right\},$$ where
\begin{itemize}
\item $(\tau_k)_{k\geq 1}$ is an increasing sequence of
$\mathbb{F}$-stopping times valued in $[t,T]$ with $\tau_{M}\leq T<
\tau_{M+1}$ for some integer-valued random variable $M<\infty$.
\item $(\alpha_k)_{k\geq 0}$ is a sequence of random variables valued in
$\{1,\cdots,d\}$ such that $\alpha_k$ is
$\mathcal{F}_{\tau_k}$-measurable, and $\alpha_0=i$.
\end{itemize}
Then the value of the following optimal switching problem
\begin{equation}\label{optimalswitch1}
y_t^{i}=\esssup_{u\in\mathcal{K}_i(t)}\mathbf{E}\left[
\int_t^{T}f^{u_s}_s(Y_s,Z_s)ds+\xi^{u_T}-\sum_{k\geq
1}C_{\tau_k}^{\alpha_{k-1},\alpha_k}\mathbf{1}_{\{t< \tau_k<
T\}}|\mathcal{F}_t\right]
\end{equation}
is given by the solution to the multidimensional reflected BSDE
(\ref{MRBSDE1}): $y_t^{i}=Y_t^{i}$ $a.s.$ for $ t\in[0,T]$. The
optimal switching strategy is given as follows: $\tau_0^*=t$,
$\alpha_0^{*}=i$ and for $k\geq 0$,
\begin{equation}\label{optimalswitching}
\tau^*_{k+1}=\inf\left\{s> \tau^*_{k}:
Y_s^{\alpha_k^*}\leq\mathcal{M}Y_s^{\alpha_k^*}\right\}\wedge T,
\end{equation}
where
$$\alpha_{k+1}^*=\argmax_{j\neq \alpha_{k}^*}\left\{
Y^{j}_{\tau_{k+1}^{*}}-C^{\alpha_k^*,j}_{\tau_{k+1}^{*}}\right\}.$$
Hence, the optimal switching strategy at any time $s\geq t$ is
$$u^{*}_s=
i\mathbf{1}_{[t,\tau_1^*]}(s)+\sum_{k=1}^{M^*}\alpha_k^{*}
\mathbf{1}_{(\tau_{k}^*,\tau_{k+1}^*]}(s),$$ where $M^*\leq M$ is
some integer-valued random variable such that $\tau^*_{M^*}\leq T<
\tau^*_{M^*+1}$.

On the other hand, Hu and Tang \cite{Hu} introduced the following
multidimensional penalized BSDE to approximate and solve the
multidimensional reflected BSDE (\ref{MRBSDE1}):
\begin{equation}\label{MPBSDE1}
Y_t^{i,\lambda}=\xi^i+\int_t^{T}f_s^i(Y_s^{\lambda},Z_s^{\lambda})ds
+\int_t^T\lambda\max\{0,\mathcal{M}Y_s^{i,\lambda}-Y_s^{i,\lambda}\}ds-\int_t^{T}Z_s^{i,\lambda}dW_s,
\end{equation}
and they proved that under Assumption \ref{Assumption2} with $f^i_s(y,z)=f^i_s(y^i,z^i)$,
$Y^{i,\lambda}$ is increasing in $\lambda$, and
\begin{equation}\label{convergence11}
\lim_{\lambda\uparrow\infty}\mathbf{E}\left[\sup_{t\in[0,T]}|Y_t^{i,\lambda}-Y_t^i|^2+\int_0^T|Z_t^{i,\lambda}-Z_t^{i}|^2dt+
\sup_{t\in[0,T]}|K_t^{i,\lambda}-K_t^i|^2\right]=0.
\end{equation}

However, the solvability of (\ref{MPBSDE1}) does not rely on the
assumption that $f^i_s(y,z)=f^i_s(y^i,z^i)$. Our aim
is therefore to give a stochastic control interpretation of the
multidimensional penalized BSDE (\ref{MPBSDE1}) with the coupled
driver $f^i_s(y,z)$, so we are at least one step closer to solve the
general optimal switching problem with the coupled driver
$f^i_s(y,z)$ is some sense.

Recall that $\{T_n\}_{n\geq 0}$ are the arrival times of the
underlying Poisson process with intensity $\lambda$,
$\mathbb{G}=\{\mathcal{G}_s\}_{s\geq 0}$ with
$\mathcal{G}_s=\mathcal{F}_s\vee\mathcal{H}_s$, and $M<\infty$ is
some integer-valued random variable such that $T_M\leq T< T_{M+1}$.

\begin{proposition}\label{proposition2}
Suppose that Assumption \ref{Assumption2} holds. Denote
$(Y^{i,\lambda},Z^{i,\lambda})$ as the unique solution to the
multidimensional penalized BSDE (\ref{MPBSDE1}). For any integer
$n\geq 1$, conditional on $\{T_{n-1}\leq t<T_{n}\}$, define the
control set $\mathcal{K}_i(\lambda,t)$ as
\begin{align*}
\mathcal{K}_i(t,\lambda)=&\ \left\{\mathbb{G}\text{-adapted\
process}\ (u_s)_{s\geq t}:
u_s=\alpha_{n-1}\mathbf{1}_{[t,\tau_{n}]}(s)+\sum_{k\geq
n}\alpha_k\mathbf{1}_{(\tau_k,\tau_{k+1}]}(s),\right.\\
&\ \ \ \text{where}\ \tau_{k}(\omega)=T_k(\omega)\ \text{for}\ n\leq
k\leq
M+1,\ \text{and}\\
&\ \ \ \left.\alpha_k\in\mathcal{G}_{T_k}\ \text{valued in}\
\{1,\cdots,d\}\ \text{with}\ \alpha_{n-1}=i.\right\}
\end{align*}
Then the value of the following optimal switching problem
\begin{equation}\label{optimalswitch2}
y_t^{i,\lambda}=\esssup_{u\in\mathcal{K}_i(t,\lambda)}\mathbf{E}\left[
\int_t^{T}f^{u_s}_s(Y^{\lambda}_s,Z^{\lambda}_s)ds+\xi^{u_T}-\sum_{k\geq
n}C_{\tau_k}^{\alpha_{k-1},\alpha_k}\mathbf{1}_{\{t< \tau_k<
T\}}|\mathcal{G}_t\right]
\end{equation}
is given by the solution of the multidimensional penalized BSDE
(\ref{MPBSDE1}): $y_t^{i,\lambda}=Y_t^{i,\lambda}$ $a.s.$. The
optimal switching strategy for (\ref{optimalswitch2}) is given as
follows: $\tau_{n-1}^*=t$, $\alpha_{n-1}^*=i$ and for $k\geq n-1$,
\begin{equation}\label{optimalswitching1.1}
\tau^*_{k+1}=\inf\left\{T_N> \tau^*_{k}:
Y_{T_N}^{\alpha_k^*,\lambda}\leq\mathcal{M}Y_{T_N}^{\alpha_k^*,\lambda}\right\}\wedge
T_{M+1},
\end{equation}
where
$$\alpha_{k+1}^*=\argmax_{j\neq \alpha_{k}^*}\left\{
Y^{j,\lambda}_{\tau_{k+1}^{*}}-C^{\alpha_k^*,j}_{\tau_{k+1}^{*}}\right\}.$$
Hence, the optimal switching strategy at any time $s\geq t$ is
$$u^{*}_s=
i\mathbf{1}_{[t,\tau_{n}^*]}(s)+\sum_{k=n}^{M^*}\alpha_k^{*}
\mathbf{1}_{(\tau_{k}^*,\tau_{k+1}^*]}(s),$$ where $M^*\leq M$ is
some integer-valued random variable such that $\tau^*_{M^*}\leq T<
\tau^*_{M^*+1}$.
\end{proposition}

\begin{proof} For any integer $n\geq 1$ and $1\leq i\leq d$, we introduce the following auxiliary optimal
stopping time problem on $\{T_{n-1}\leq t<T_n\}$:
\begin{equation}\label{optimalswitch2.1}
\tilde{y}_{t}^{i,\lambda}=\esssup_{\tau\in\mathcal{R}_{T_{n}}(\lambda)}
\mathbf{E}\left[\int_{t}^{\tau\wedge
T}f_s^i(Y_s^{\lambda},Z_s^{\lambda})ds+\mathcal{M}Y^{i,\lambda}_{\tau}\mathbf{1}_{\{\tau<
T\}}+\xi^i\mathbf{1}_{\{\tau\geq T\}}|\mathcal{G}_{t}\right].
\end{equation}
From Theorem \ref{Theorem} (and Remark \ref{remark}), we know that
its value is given by $\tilde{y}^{i,\lambda}_{t}=Y^{i,\lambda}_{t}$
$a.s.$, and the optimal stopping time is given by
$$\tau^*_{T_{n}}=\inf\left\{T_N\geq T_{n}: Y^{i,\lambda}_{T_N}\leq
\mathcal{M}Y^{i,\lambda}_{T_N}\right\}
\wedge T_{M+1}.$$

Now for any switching strategy $u\in\mathcal{K}_i(t,\lambda)$ with
the form
\begin{equation*}
u_s=i\mathbf{1}_{[t,T_{n}]}(s)+\sum_{k=n
}^{M}\alpha_{k}\mathbf{1}_{(T_{k},T_{k+1}]}(s),
\end{equation*}
we consider the auxiliary optimal stopping problem
(\ref{optimalswitch2.1}) stopping at the Poisson arrival time $T_n$,
and switching to $\alpha_n$,
\begin{equation}\label{auxilequ1}
\tilde{y}^{i,\lambda}_t\geq \mathbf{E}\left[\int_t^{T_n\wedge
T}f_s^i(Y_s^{\lambda},Z_s^{\lambda})ds+(Y^{\alpha_n,\lambda}_{T_n}-C^{i,\alpha_n}_{T_n})
\mathbf{1}_{\{T_n<T\}}+\xi^i\mathbf{1}_{\{T_n\geq
T\}}|\mathcal{G}_t\right].
\end{equation}

From Theorem \ref{Theorem}, $Y^{\alpha_n,\lambda}_{T_n}$ is the
value of the optimal stopping problem (\ref{optimalswitch2.1})
starting from $T_n$. We consider such an optimal stopping problem
stopping at the Poisson arrival time $T_{n+1}$, and switching to
$\alpha_{n+1}$,
\begin{align}\label{auxilequ2}
&Y^{\alpha_n,\lambda}_{T_n}=\tilde{y}_{T_n}^{\alpha_n,\lambda}\\
&\geq \mathbf{E}\left[\int_{T_n}^{T_{n+1}\wedge
T}f_s^{\alpha_n}(Y_s^{\lambda},Z_s^{\lambda})ds+(Y^{\alpha_{n+1},\lambda}_{T_{n+1}}
-C^{\alpha_{n},\alpha_{n+1}}_{T_{n+1}})
\mathbf{1}_{\{T_{n+1}<T\}}+\xi^{\alpha_n}\mathbf{1}_{\{T_{n+1}\geq
T\}}|\mathcal{G}_{T_n}\right].\nonumber
\end{align}
By plugging (\ref{auxilequ2}) into (\ref{auxilequ1}), we have
\begin{align*}
&\tilde{y}_{t}^{i,\lambda}\geq\ \mathbf{E}\left[\int_t^{T_n\wedge
T}f_s^i(Y_s^{\lambda},Z_s^{\lambda})ds+\int_{T_n\wedge
T}^{T_{n+1}\wedge
T}f_s^{\alpha_n}(Y_s^{\lambda},Z_s^{\lambda})ds-C_{T_n}^{i,\alpha_n}\mathbf{1}_{\{T_n<T\}}\right.\\
&\left.-
C_{T_{n+1}}^{\alpha_n,\alpha_{n+1}}\mathbf{1}_{\{T_{n+1}<T\}}+\xi^i\mathbf{1}_{\{T_n\geq
T\}}+\xi^{\alpha_{n}}\mathbf{1}_{\{T_n<T\leq T_{n+1}\}}
+Y_{T_{n+1}}^{\alpha_{n+1},\lambda}\mathbf{1}_{\{T_{n+1}<T\}}|\mathcal{G}_t\right].
\end{align*}

We repeat the above procedure $M$ times, and obtain
\begin{align*}
&\tilde{y}_{t}^{i,\lambda}\geq\mathbf{E}\left[\int_t^{T_n\wedge
T}f_s^i(Y_s^{\lambda},Z_s^{\lambda})ds
+\xi^i\mathbf{1}_{\{T_n\geq T\}}\right.\\
&\left.+\sum_{k=n}^{M}\left(\int_{T_k\wedge T}^{T_{k+1}\wedge
T}f_s^{\alpha_{k}}(Y_s^{\lambda},Z_s^{\lambda})ds-C_{T_{k}}^{\alpha_{k}-1,\alpha_{k}}\mathbf{1}_{\{T_{k}<T\}}
+\xi^{\alpha_k}\mathbf{1}_{\{T_{k}< T\leq
T_{k+1}\}}\right)|\mathcal{G}_t\right].
\end{align*}
Since $T_M\leq T<T_{M+1}$, the above inequality is  further
simplified to
$$\tilde{y}^{i,\lambda}_t\geq\mathbf{E}\left[
\int_t^{T}f^{u_s}_s(Y_s^{\lambda},Z_s^{\lambda})ds+\xi^{u_T}-\sum_{k\geq
n}C_{T_k}^{\alpha_{k-1},\alpha_k}\mathbf{1}_{\{t< T_k<
T\}}|\mathcal{G}_t\right].$$ By taking the supremum over
$u\in\mathcal{K}_i(t,\lambda)$ and using Theorem \ref{Theorem} once
again, we prove that on $\{T_{n-1}\leq t<T_n\}$,
$$Y_t^{i,\lambda}=\tilde{y}^{i,\lambda}_t\geq
y^{i,\lambda}_t.$$

To prove the reverse inequality, we take the switching strategy
$u=u^*$. From Theorem
\ref{Theorem} (and Remark \ref{remark}), $\tau_n^*$ is the optimal
stopping time for (\ref{optimalswitch2.1}). By the definition of
$\alpha_{n}^*$,
$$\mathcal{M}Y^{i,\lambda}_{\tau_n^*}=
\max_{j\neq
i}\{Y_{\tau_n^*}^{j,\lambda}-C_{\tau_n^*}^{i,j}\}=Y^{\alpha_n^*,\lambda}_{\tau_n^*}-
C_{\tau_n^*}^{i,\alpha_{n}^*}.$$ Therefore,
\begin{equation}\label{auxilequ3}
\tilde{y}^{i,\lambda}_t=\mathbf{E}\left[\int_t^{\tau_n^*\wedge
T}f_s^i(Y_s^{\lambda},Z_s^{\lambda})ds+\left(Y^{\alpha_n^*,\lambda}_{\tau_n^*}-C^{i,\alpha_n^*}_{\tau_n^*}\right)
\mathbf{1}_{\{\tau_n^*<T\}}+\xi^i\mathbf{1}_{\{\tau_n^*\geq
T\}}|\mathcal{G}_t\right].
\end{equation}
Similarly, $\tau^*_{n+1}$ is the optimal stopping time for
(\ref{optimalswitch2.1}) starting from $\tau_n^*$, and
$Y_{\tau_n^*}^{\alpha_n^*,\lambda}=\tilde{y}_{\tau_{n}^*}^{\alpha_n^*,\lambda}$.
By the definition of $\alpha_{n+1}^*$,
$$\mathcal{M}Y^{\alpha_n^*,\lambda}_{\tau_{n+1}^*}=
\max_{j\neq
\alpha_n^*}\left\{Y_{\tau_{n+1}^*}^{j,\lambda}-C_{\tau_{n+1}^*}^{\alpha_n^*,j}\right\}=
Y^{\alpha_{n+1}^*,\lambda}_{\tau_{n+1}^*}-C_{\tau_{n+1}^*}^{\alpha_{n}^*,\alpha_{n+1}^*}.$$
Hence,
\begin{align}\label{auxilequ4}
&Y_{\tau_n^*}^{\alpha_n^*,\lambda}=\tilde{y}^{\alpha_n^*,\lambda}_{\tau_n^*}\\
&=\mathbf{E}\left[\int_{\tau_n^*}^{\tau_{n+1}^*\wedge
T}f_s^{\alpha_n^*}(Y_s^{\lambda},Z_s^{\lambda})ds+\left(Y^{\alpha_{n+1}^*,\lambda}_{\tau_{n+1}^*}-
C^{\alpha_n^*,\alpha_{n+1}^*}_{\tau_{n+1}^*}\right)
\mathbf{1}_{\{\tau_{n+1}^*<T\}}+\xi^{\alpha_n^*}\mathbf{1}_{\{\tau_{n+1}^*\geq
T\}}|\mathcal{G}_{\tau_n^*}\right]\nonumber.
\end{align}
Plugging (\ref{auxilequ4}) into (\ref{auxilequ3}) gives us
\begin{align*}
&\tilde{y}_{t}^{i,\lambda}=\mathbf{E}\left[\int_t^{\tau_n^*\wedge
T}f_s^i(Y_s^{\lambda},Z_s^{\lambda})ds+\int_{\tau_n^*\wedge
T}^{\tau^*_{n+1}\wedge
T}f_s^{\alpha_n^*}(Y_s^{\lambda},Z_s^{\lambda})ds-
C_{\tau_n^*}^{i,\alpha_n^*}\mathbf{1}_{\{\tau_n^*<T\}}\right.\\
&\left.-C_{\tau_{n+1}^*}^{\alpha_n^*,\alpha_{n+1}^*}\mathbf{1}_{\{\tau_{n+1}^*<T\}}
+\xi^i\mathbf{1}_{\{\tau_n^*\geq
T\}}+\xi^{\alpha_{n}}\mathbf{1}_{\{\tau_n^*<T\leq \tau_{n+1}^*\}}
+Y_{\tau_{n+1}^*}^{\alpha_{n+1}^*,\lambda}\mathbf{1}_{\{\tau_{n+1}^*<T\}}|\mathcal{G}_t\right].
\end{align*}

We repeat the above procedure $M^*$ times, and obtain
\begin{align*}
&\tilde{y}_{t}^{i,\lambda}=\mathbf{E}\left[\int_t^{\tau^*_n\wedge
T}f_s^i(Y_s^{\lambda},Z_s^{\lambda})ds
+\xi^i\mathbf{1}_{\{\tau^*_n\geq T\}}\right.\\
&\left.+\sum_{k=n}^{M^*}\left(\int_{\tau_k^*\wedge
T}^{\tau^*_{k+1}\wedge
T}f_s^{\alpha_{k}^*}(Y_s^{\lambda},Z_s^{\lambda})ds-C_{\tau^*_{k}}^{\alpha_{k-1}^*,\alpha_{k}^*}\mathbf{1}_{\{\tau^*_{k}<T\}}
+\xi^{\alpha_k^*}\mathbf{1}_{\{\tau^*_{k}<
T\leq \tau_{k+1}^*\}}\right)|\mathcal{G}_t\right]\\
&\ \ \ \ =\mathbf{E}\left[
\int_t^{T}f^{u_s^*}_s(Y_s^{\lambda},Z_s^{\lambda})ds+\xi^{u_T^*}-\sum_{k\geq
n}C_{\tau^*_k}^{\alpha_{k-1}^*,\alpha_k^*}\mathbf{1}_{\{t< \tau_k^*<
T\}}|\mathcal{G}_t\right]\leq y^{i,\lambda}_t.
\end{align*}
Theorem \ref{Theorem} then implies that
$$Y_t^{i,\lambda}=\tilde{y}_t^{i,\lambda}\leq y_t^{i,\lambda},$$
and $u^*$ is the optimal switching strategy.
\end{proof}

\begin{remark}
The optimal switching representation (\ref{optimalswitch2}) of the multidimensional
penalized BSDE (\ref{MPBSDE1}) has a natural economic application to
the menu cost model of Stokey \cite{Stokey}, which allows the
occasional arrival of opportunities to adjust without paying the
fixed cost, and those opportunities are modeled as Poisson arrivals.
See also \cite{LiangWei} for an extension to an infinite horizon BSDE
setting with the analysis of the corresponding free boundaries in the sense of
Ly Vath and Pham \cite{Pham}.
\end{remark}\\

The other commonly used penalization scheme for the multidimensional
reflected BSDE (\ref{MRBSDE1}) is the following equation:
\begin{equation}\label{MPBSDE2}
Y_t^{i,\lambda}=\xi^i+\int_t^{T}f_s^i(Y_s^{\lambda},Z_s^{\lambda})ds
+\int_t^T\sum_{j=1}^d\lambda\max\{0,Y_s^{j,\lambda}-
C_s^{i,j}-Y_s^{i,\lambda}\}ds-\int_t^{T}Z_s^{i,\lambda}dW_s,
\end{equation}
and we still have the convergence (\ref{convergence11}) as shown in
\cite{Hu}.

In the following, we show that (\ref{MPBSDE2}) is closely related to
the BSDE with regime switching on a Markov chain. Regime switching
on Markov chains has been found useful in many applications as shown
in \cite{Zhang1} and \cite{Zhang2}. Its application in BSDE can be
found in a recent work \cite{Wu} among others.

Define a Markov chain $(X_t)_{t\geq 0}$ with state space
$\{1,2,\dots,d\}$, and its $Q$-matrix: $q_{ij}=\lambda$ if $i\neq
j$, and $q_{ij}=-(d-1)\lambda$ if $i=j$. The jump times are denoted
as $\{T_n\}_{n\geq 1}$. At each jump time $T_n$, the player has the
right to choose if switching from the current state or not, and if she
switches, a cost $C_{T_n}^{i,j}$ incurs if the Markov chain jumps
from the state $i$ to $j$.

For any integer $n\geq 1$, conditional on $\{T_{n-1}\leq t<T_n\}$,
define the following control set:
\begin{align*}
\mathcal{K}_i(t,Q)=&\ \left\{\mathbb{G}\text{-adapted\ process}\
(u_s)_{s\geq t}:
u_s=X_{\tau_{n-1}}\mathbf{1}_{[t,\tau_{n}]}(s)+\sum_{k\geq
n}X_{\tau_k}\mathbf{1}_{(\tau_k,\tau_{k+1}]}(s),\right.\\
&\ \ \ \left.\text{where}\ (\tau_{k})_{k\geq n}\ \text{chosen\
from}\ T_N\ \text{for}\ {n\leq N\leq M+1},\ \text{and}\
X_{\tau_{n-1}}=i.\right\}
\end{align*}

\begin{proposition}\label{proposition22}
Suppose that Assumption \ref{Assumption2} holds. Denote
$(Y^{i,\lambda},Z^{i,\lambda})$ as the unique solution to the
multidimensional penalized BSDE (\ref{MPBSDE2}). Then the value of
the following optimal switching problem
\begin{equation}\label{optimalswitch22}
y_t^{i,\lambda}=\esssup_{u\in\mathcal{K}_i(t,Q)}\mathbf{E}\left[
\int_t^{T}f^{u_s}_s(Y^{\lambda}_s,Z^{\lambda}_s)ds+\xi^{u_T}-\sum_{k\geq
n}C_{\tau_k}^{X_{\tau_{k-1}},X_{\tau_k}}\mathbf{1}_{\{t< \tau_k<
T\}}|\mathcal{G}_t\right]
\end{equation}
is given by the solution of the multidimensional penalized BSDE
(\ref{MPBSDE2}): $y_t^{i,\lambda}=Y_t^{i,\lambda}$ $a.s.$. The
optimal switching strategy for (\ref{optimalswitch2}) is given as
follows: $\tau_{n-1}^*=t$, and for $k\geq n-1$,
\begin{equation}\label{optimalswitching2.2}
\tau^*_{k+1}=\inf\left\{T_N> \tau^*_{k}:
Y_{T_N}^{X_{\tau_k^*},\lambda}\leq
Y_{T_N}^{X_{T_N},\lambda}-C_{T_N}^{X_{\tau_k^*},X_{T_N}}\right\}\wedge T_{M+1}.
\end{equation}
Hence, the optimal switching strategy at any time $s\geq t$ is
$$u^{*}_s=
i\mathbf{1}_{[t,\tau_{n}^*]}(s)+\sum_{k=n}^{M^*}X_{\tau_k^{*}}
\mathbf{1}_{(\tau_{k}^*,\tau_{k+1}^*]}(s),$$ where $M^*\leq M$ is
some integer-valued random variable such that $\tau^*_{M^*}\leq T<
\tau^*_{M^*+1}$.
\end{proposition}

\begin{proof}
We first rewrite (\ref{MPBSDE2}) in terms of $q_{ij}$ as follows,
\begin{equation*}
Y_t^{i,\lambda}=\xi^i+\int_t^{T}f_s^i(Y_s^{\lambda},Z_s^{\lambda})ds
+\int_t^T\sum_{j=1}^dq_{ij}\max\{0,Y_s^{j,\lambda}-
C_s^{i,j}-Y_s^{i,\lambda}\}ds-\int_t^{T}Z_s^{i,\lambda}dW_s.
\end{equation*}
Then similar to Lemma \ref{lemma0}, we have that
\begin{align}\label{DPEforBSDE222}
Y_{t}^{i,\lambda}=&\mathbf{E}\left[\int_{t}^{T_{n}\wedge
T}f_s^i(Y_s^{\lambda},Z_s^{\lambda})ds\right.\\
&\left.+\max\left\{Y_{T_n}^{X_{T_n},\lambda}-C_{T_n}^{i,X_{T_n}},
Y^{i,\lambda}_{T_{n}}\right\}
\mathbf{1}_{\{T_{n}\leq T\}}+\xi^i\mathbf{1}_{\{T_{n}>
T\}}|\mathcal{G}_{t}\right].\nonumber
\end{align}
conditional on $\{T_{n-1}\leq t<T_n\}$. From Theorem \ref{Theorem}, $Y_t^{i,\lambda}$
is the value of the following optimal stopping time problem:
\begin{equation}\label{optimalswitch3.1}
\esssup_{\tau\in\mathcal{R}_{T_{n}}(Q)}
\mathbf{E}\left[\int_{t}^{\tau\wedge
T}f_s^i(Y_s^{\lambda},Z_s^{\lambda})ds+\left(Y_{\tau}^{X_{\tau},\lambda}-C_{\tau}^{i,X_{\tau}}\right)\mathbf{1}_{\{\tau<
T\}}+\xi^i\mathbf{1}_{\{\tau\geq T\}}|\mathcal{G}_{t}\right],
\end{equation}
where
$$\mathcal{R}_{T_n}{(Q)}=\left\{\mathbb{G}\text{-stopping\ time}\ \tau\ \text{for}\ \tau(\omega)=T_N(\omega)
\ \text{where}\ n\leq N\leq M+1.\right\}$$
with the optimal stopping time given by
$$\tau^*_{T_{n}}=\inf\left\{T_N\geq T_{n}: Y^{i,\lambda}_{T_N}\leq
Y^{X_{T_N},\lambda}_{T_N}-C_{T_N}^{i,X_{T_N}}\right\}
\wedge T_{M+1}.$$
The rest of the proof is then similar to that of Proposition \ref{proposition2}, so we omit it.
\end{proof}

\begin{remark} If we compare between the optimal switching representations (\ref{optimalswitch2})
and (\ref{optimalswitch22}), the former only allows the player to choose the switching regimes on a sequence of Poisson arrival times, while
the latter only allows the player to choose the switching times with the regimes following a Markov chain.
\end{remark}

\section{Application VI: Constrained Reflected BSDE}\label{Sec_Z}

In Cvitanic et al \cite{Cvitanic}, the authors introduced a new
class of BSDEs with a convex constraint on the \emph{hedging} process
$Z$, and solved the equation using the stochastic control
method\footnote{I would like to thank Ioannis Karatzas for the
suggestion of this section.}. Their equation was further generalized
by Peng \cite{Peng}, and in particular, by Peng and Xu \cite{Xu22}
to reflected BSDE with a general constraint on $Z$
(\emph{constrained reflected BSDE} for short), where the monotonic
limit theorem was introduced in order to show the associated
penalized equation converges to the constrained reflected BSDE. A
constraint reflected BSDE has the form
\begin{equation}\label{constriantBSDE}
Y_t=\xi+\int_t^Tf_s(Y_s,Z_s)ds+\int_t^TdK_s^Y+\int_t^TdK_s^Z-\int_t^TZ_sdW_s
\end{equation}
under the constraints
\begin{align*}
\text{(Dominating Condition)}:\ \ \ & Y_t\geq S_t\ \text{for}\ t\in[0,T],\\
\text{(Skorohod Condition)}:\ \ \ & \int_0^T(Y_t-S_t)dK_t^Y=0\
\text{for}\ K^Y\ \text{continuous and increasing},\\
\text{(Hedging Constraint)}:\ \ \ & Z_t\in\Gamma\ \text{for}\
t\in[0,T].
\end{align*}
The terminal data $\xi$, the driver $f_s(y,z)$, the obstacle
$(S_t)_{0\leq t\leq T}$, and the constraint set
$\Gamma\subset\mathbb{R}^d$ are the given data. A solution to the
constrained reflected BSDE (\ref{constriantBSDE}) is a quadruple of
$\mathbb{F}$-adapted processes $(Y,Z,K^Y,K^Z)$, where $K^Y$ is used
to pushed up the solution $Y$ in order to satisfy the dominating
condition, and $K^Z$ (RCLL and increasing) is used to enforce the
solution $Z$ staying in the constraint set $\Gamma$.

The following standard assumption on the data set $(\xi,f,S,\Gamma)$
is imposed as in Peng and Xu \cite{Xu22}, so (\ref{constriantBSDE})
admits a smallest solution $(Y,Z,K^Y,K^Z)$, in the sense that if
$(\overline{Y},\overline{Z},\overline{K^Y},\overline{K^Z})$ is
another solution to (\ref{constriantBSDE}), then $\overline{Y}_t\geq
Y_t$ $a.s.$ for $t\in[0,T]$.

\begin{assumption}\label{Assumption3}
\begin{itemize}
\item The terminal data $\xi$, the driver $f_s(y,z)$, and the
obstacle $S$ satisfy Assumption \ref{Assumption};
\item The set $\Gamma$ is a closed and convex set in $\mathbb{R}^d$
including the origin;
\item There exists at least one solution $(\overline{Y},\overline{Z},\overline{K^Y},\overline{K^Z})$ to
(\ref{constriantBSDE}).
\end{itemize}
\end{assumption}

When the driver $f_s(y,z)$ is independent of $(y,z)$, denoted as
$f_s$ in such a situation, Cvitanic et al \cite{Cvitanic} gave a
stochastic control representation for the solution of the
constrained reflected BSDE (\ref{constriantBSDE}). Indeed, define
the control set $\mathcal{D}(t)$ as
\begin{align*}
\mathcal{D}(t)=&\bigcup_{m\geq 1}\left\{ \mathbb{F}\text{-adapted\
process}\ (\nu_s)_{s\geq t}: \ \mathbb{H}^2\text{-square
integrable},\ \text{valued\ in}\
\Gamma^*,\right.\\
&\ \ \ \ \ \ \ \left.\text{and}\ |\nu_s|\leq m\ \text{for}\
s\in[t,T]\right\}.
\end{align*}
The valued set $\Gamma^*$ is defined as follows: Given the closed
and convex set $\Gamma$, define its support function
$\delta_{\Gamma}^*(\cdot)$ as the convex dual of the characteristic
function $\delta_{\Gamma}(\cdot)$ of $\Gamma$,
$$\delta_{\Gamma}^*(z)=\sup_{\bar{z}\in\mathbb{R}^d}\left\{\bar{z}\cdot z-\delta_{\Gamma}(\bar{z})\right\},$$
which is bounded on compact subsets of the barrier cone $\Gamma^*$,
$$\Gamma^*=\left\{z\in\mathbb{R}^d: \delta_{\Gamma}^*(z)<\infty\right\}.$$
Given $\nu\in\mathcal{D}(t)$, define an equivalent probability
measure $\mathbf{P}^{\nu}$ as
$$\frac{d\mathbf{P}^{\nu}}{d\mathbf{P}}=\exp\left\{\int_0^{\cdot}\nu_sdW_s-\frac{1}{2}\int_0^{\cdot}|\nu_s|^2ds\right\}.$$
Then the value of the following stochastic control problem
\begin{equation}\label{constraint_representation}
y_t=\esssup_{\tau\in\mathcal{R}(t),\nu\in\mathcal{D}(t)}\mathbf{E}^{\mathbf{P}^{\nu}}
\left[\int_t^{\tau\wedge
T}[f_s-\delta^*_{\Gamma}(\nu_s)]ds+S_{\tau}\mathbf{1}_{\{\tau<T\}}+
\xi\mathbf{1}_{\{\tau\geq T\}}|\mathcal{F}_t\right]
\end{equation}
is given by the solution to the constrained reflected BSDE
(\ref{constriantBSDE}) with the driver $f_s$: $y_t=Y_t$ $a.s.$ for
$t\in[0,T]$.

On the other hand, (\ref{constriantBSDE}) can be solved by
approximating two ``local time'' processes $K^Y$ and $K^Z$ by
$$K^{Y,\lambda}_t=\int_0^t\lambda\max\{0,S_s-Y_s^{(\lambda,m)}\}ds$$
and
$$K^{Z,m}_t=\int_0^tm\times\text{dist}_{\Gamma}(Z^{(\lambda,m)}_s)ds=
\int_0^tm\times\inf_{z\in\Gamma}|z-Z_s^{(\lambda,m)}|ds$$
respectively, where $(Y^{(\lambda,m)},Z^{(\lambda,m)})$ is the
solution of the following constrained penalized BSDE
\begin{align}\label{constriantpendBSDE}
Y^{(\lambda,m)}_t=&\
\xi+\int_t^Tf_s(Y_s^{(\lambda,m)},Z_s^{\lambda,m})
+\lambda\max\{0,S_s-Y_s^{(\lambda,m)}\}+m\times\text{dist}_{\Gamma}(Z^{(\lambda,m)}_s)ds\\
&\ -\int_t^TZ_s^{(\lambda,m)}dW_s.\nonumber
\end{align}
Peng and Xu \cite{Xu22} proved that the solution
$(Y^{(\lambda,m)},Z^{(\lambda,m)},K^{Y,\lambda},K^{Z,m})$ converges
to the smallest solution $(Y,Z,K^Y,K^Z)$ of the constrained
reflected BSDE (\ref{constriantBSDE}) in the sense of monotonic
limit theorem as $\lambda,m\uparrow\infty$.

Our aim in this section is to give a stochastic control
representation of the constrained penalized BSDE
(\ref{constriantpendBSDE}), which has a similar structure to the
stochastic control representation (\ref{constraint_representation}).

\begin{proposition}\label{proposition3}
Suppose that Assumption \ref{Assumption3} holds. Denote
$(Y^{(\lambda,m)},Z^{(\lambda,m)})$ as the unique solution to the
constrained penalized BSDE (\ref{constriantpendBSDE}). For any
$t\in[0,T]$, define the control set $\mathcal{D}(t,m)$ as
\begin{align*}
\mathcal{D}(t,m)=&\left\{ \mathbb{F}\text{-adapted\ process}\
(\nu_s)_{s\geq t}: \ \mathbb{H}^2\text{-square integrable},\
\text{valued\ in}\ \Gamma^*\right.\\
&\left.\text{and}\ |\nu_s|\leq m\ \text{for}\ s\in[t,T]\right\},
\end{align*}
and for any integer $i\geq 1$, the control set
$\mathcal{R}_{T_i}(\lambda)$ as in Theorem \ref{Theorem}. Then
conditional on $\{T_i\leq t<T_i\}$, the value of the following
stochastic control problem
\begin{align}\label{constraint_representation_penality}
y_t^{(\lambda,m)}=&\esssup_{\tau\in\mathcal{R}_{T_i}(\lambda),
\nu\in\mathcal{D}(t,m)}\mathbf{E}^{\mathbf{P}^{\nu}}
\left[\int_t^{\tau\wedge
T}[f_s(Y_s^{(\lambda,m)},Z_s^{(\lambda,m)})-\delta^*_{\Gamma}(\nu_s)]ds\right.\\
&\ \ \ \ \ \ \ \ \ \ \ \ \ \ \ \ \ \ \ \ \ \ \ \ \ \ \ \ \
\left.+S_{\tau}\mathbf{1}_{\{\tau<T\}}+ \xi\mathbf{1}_{\{\tau\geq
T\}}|\mathcal{G}_t\right]\nonumber
\end{align}
is given by the solution to the constrained penalized BSDE
(\ref{constriantpendBSDE}): $y_t^{(\lambda,m)}=Y_t^{(\lambda,m)}$
$a.s.$. The optimal stopping time
$\tau^*_{T_i}\in\mathcal{R}_{T_i}(\lambda)$ is given by
\begin{equation}\label{optimaltime}
\tau^*_{T_i}=\inf\{T_{N}\geq T_i: Y^{(\lambda,m)}_{T_{N}}\leq
S_{T_{N}}\}\wedge T_{M+1},
\end{equation}
and the optimal control $v^*\in\mathcal{D}(t,m)$ is the solution of
the following algebraic equation
\begin{equation}\label{algebraequ0}
m\times\text{dist}_{\Gamma}(Z^{(\lambda,m)}_s)=Z^{(\lambda,m)}_s\cdot
v_s^*-\delta_{\Gamma}^*(v_s^*),\ \text{for}\ a.e.\
(s,\omega)\in[t,T]\times\Omega.
\end{equation}
\end{proposition}

\begin{proof} We only consider the linear case $f_s(y,z)=f_s$, as the
proof for the nonlinear case $f_s(y,z)$ is the same as the one in
Section \ref{sec2.2}.

First, we remark that if $v\in\mathcal{D}(t,m)$, in particular
$|v_s|\leq m$, then the support function $\delta^*_{\Gamma}(v_s)$
has the convex dual representation
$$\delta^*_{\Gamma}(v_s)=\sup_{z\in\mathbb{R}^d}\left\{z\cdot v_s-
m\times\text{dist}_{\Gamma}(z)\right\},\ \text{for}\ a.e.\
(s,\omega)\in[t,T]\times\Omega.$$ See Lemma 3.1 in \cite{Cvitanic}
for the proof. Intuitively, it means that we use
$m\times\text{dist}_{\Gamma}(\cdot)$ to approximate the
characteristic function $\delta_{\Gamma}(\cdot)$. Moreover, as shown
in \cite{Cvitanic}, since $m\times\text{dist}_{\Gamma}(\cdot)$ is
convex,
$$m\times\text{dist}_{\Gamma}(Z_s^{(\lambda,m)})
=\sup_{\nu\in\mathcal{D}(t,m)}\{Z_s^{(\lambda,m)}\cdot\nu_s-\delta_{\Gamma}^*(\nu_s)\},\
\text{for}\ a.e.\ (s,\omega)\in[t,T]\times\Omega,$$ and there exists
$v^*\in\mathcal{D}(t,m)$ solving the algebraic equation
(\ref{algebraequ0}).

Now for any control $\nu\in\mathcal{D}(t,m)$, we rewrite
(\ref{constriantpendBSDE}) as
\begin{align}
Y_t^{(\lambda,m)}=&\ \xi+ \int_t^{
T}\left[f_s+\lambda\max\{0,S_s-Y_s^{(\lambda,m)}\}-\delta_{\Gamma}^*(\nu_s)+Z^{(\lambda,m)}_s\cdot
\nu_s\right]ds\\
&\ +\int_t^{
T}\left[m\times\text{dist}_{\Gamma}(Z^{(\lambda,m)}_s)-Z^{(\lambda,m)}_s\cdot
\nu_s+\delta^*_{\Gamma}(\nu_s)\right]ds-
\int_t^{T}Z^{(\lambda,m)}_sdW_s.\nonumber
\end{align}

Since
$$m\times\text{dist}_{\Gamma}(Z^{(\lambda,m)}_s)-Z^{(\lambda,m)}_s\cdot
\nu_s+\delta^*_{\Gamma}(\nu_s)\geq 0$$ for any
$\nu\in\mathcal{D}(t,m)$, from the BSDE comparison theorem,
$Y_t^{(\lambda,m)}\geq Y_t^{(\lambda,m)}(\nu)$, where
$Y^{(\lambda,m)}(\nu)$ is the solution of the following BSDE
\begin{align*}
Y_t^{(\lambda,m)}(\nu)=&\ \xi+ \int_t^{
T}\left[f_s+\lambda\max\{0,S_s-Y_s^{(\lambda,m)}(\nu)\}-\delta_{\Gamma}^*(\nu_s)+Z^{(\lambda,m)}_s(\nu)\cdot
\nu_s\right]ds\\
&-\int_t^{T}Z^{(\lambda,m)}_s(\nu)dW_s,
\end{align*}
or equivalently, under the probability measure $\mathbf{P}^{\nu}$,
\begin{equation*}
Y_t^{(\lambda,m)}(\nu)=\xi+ \int_t^{
T}\left[f_s+\lambda\max\{0,S_s-Y_s^{(\lambda,m)}(\nu)\}-\delta_{\Gamma}^*(\nu_s)\right]ds-
\int_t^{T}Z^{(\lambda,m)}_s(\nu)dW_s^{\nu},
\end{equation*}
where $W^{\nu}_s=W_s-\int_0^sv_udu$ for $s\geq 0$ is the Brownian
motion under the probability measure $\mathbf{P}^{\nu}$.

From Theorem \ref{Theorem}, we know that conditional on
$\{T_{i-1}\leq t<T_i\}$, $Y_t^{(\lambda,m)}(\nu)\geq
y_t^{(\lambda,m)}(\tau,\nu)$ for any stopping time
$\tau\in\mathcal{R}_{T_i}(\lambda)$, where
\begin{equation}\label{linear1}
y_t^{(\lambda,m)}(\tau,\nu)=\mathbf{E}^{\mathbf{P}^{\nu}}
\left[\int_t^{\tau\wedge
T}[f_s-\delta^*_{\Gamma}(\nu_s)]ds+S_{\tau}\mathbf{1}_{\{\tau<T\}}+
\xi\mathbf{1}_{\{\tau\geq T\}}|\mathcal{G}_t\right].
\end{equation}
Hence, $Y_t^{(\lambda,m)}\geq y_t^{(\lambda,m)}(\tau,\nu)$. Taking
the supremum over $\tau\in\mathcal{R}_{T_i}(\lambda)$ and
$\nu\in\mathcal{D}(t,m)$ gives us $Y_t^{(\lambda,m)}\geq
y_t^{(\lambda,m)}$ on $\{T_{i-1}\leq t<T_i\}$.

Next, we choose $\nu=\nu^*$ and $\tau=\tau^*_{T_i}$ to get the
reverse inequality. Indeed, for $v^*$ solving (\ref{algebraequ0}),
$Y_t^{(\lambda,m)}=Y_t^{(\lambda,m)}(\nu^*)$. Moreover, if we choose
$\nu=\nu^*$ and $\tau=\tau^*_{T_i}$, we get
$$\tau^*_{T_i}=\inf\{T_{n}\geq T_i: Y^{(\lambda,m)}_{T_{n}}\leq
S_{T_{n}}\}\wedge T_{M+1}=\inf\{T_{n}\geq T_i:
Y^{(\lambda,m)}_{T_{n}}(\nu^*)\leq S_{T_{n}}\}\wedge T_{M+1}.$$ From
Theorem \ref{Theorem},
$Y_t^{(\lambda,m)}(\nu^*)=y_t^{(\lambda,m)}(\tau^*_{T_i},\nu^*)\leq
y_t^{(\lambda,m)}$. Therefore, $Y_t^{(\lambda,m)}=y_t^{(\lambda,m)}$
on $\{T_{i-1}\leq t<T_i\}$, and $(\nu^*,\tau^*_{T_i})$ are the
optimal control and optimal stopping time of
(\ref{constraint_representation_penality}) respectively.
\end{proof}

\section{Conclusion}

In this paper, we find the stochastic control representations of
(multidimensional, constrained) reflected BSDEs and associated
penalized BSDEs, which are summarized in the following table. The
main feature of the related optimal stopping representation is that
the player only stops at arrival times of some exogenous Poisson
process.

\begin{table}[[!htb]\small\caption{\small Stochastic Control Representations of
Reflected BSDEs and Penalized BSDEs}
\begin{center}
\begin{tabular}{ccc}
\hline\hline
&Stochastic control representations\\
\hline
Reflected BSDE& (\ref{OptimalStopping1}) with $\tau\in\mathcal{R}(t)$/(\ref{optimalcontrol2}) with $r\in\mathcal{A}(t)$\\
Penalized BSDE& (\ref{OptimalStopping2}) with $\tau\in\mathcal{R}_{T_i}(\lambda)$/ (\ref{optimalcontrol}) with $r\in\mathcal{A}(t,\lambda)$\\
Multidimensional Reflected BSDE& (\ref{optimalswitch1}) with $u\in\mathcal{K}_i(t)$\\
Multidimensional Penalized BSDE&(\ref{optimalswitch2}) with $u\in\mathcal{K}_i(t,\lambda)$/(\ref{optimalswitch22}) with $u\in\mathcal{K}_{i}(t,Q)$\\
Constrained Reflected BSDE & (\ref{constraint_representation}) with $\tau\in\mathcal{R}(t)$ and $\nu\in\mathcal{D}(t)$\\
Constrained Penalized BSDE &
(\ref{constraint_representation_penality}) with
$\tau\in\mathcal{R}_{T_i}(t,\lambda)$ and $\nu\in\mathcal{D}(t,m)$\\
\hline
\end{tabular}
\end{center}
\end{table}

Finally, it seems that the only existing representation result for
penalized BSDE was given by Lepeltier and Xu in \cite{Xu1} and
\cite{Xu2}\footnote{I would like to thank Mingyu Xu for providing me
with these two references.}, where they found a connection between
penalized BSDE and a standard optimal stopping problem with modified
obstacle $\min\{S_t,Y^{\lambda}_t\}$. Our represent results are
different, and seem more natural: Penalized BSDE is nothing but a
random time discretization of the optimal stopping representation
for the corresponding reflected BSDE, where the time is discretized
by Poisson arrival times.

\small
\section*{Acknowledgments}
The author would like to thank the editor Prof.Qing Zhang, an associate editor,
and a referee for their valuable comments and suggestions on the manuscript, and for
their patient handling of the paper. The author is grateful to  Erhan Bayraktar, David Hobson, Ying Jiao, Ioannis
Karatzas, Shige Peng and Mingyu Xu for helpful discussions, and
especially to Sam Cohen for pointing out a mistake in an early
version. The author also thanks participants in seminars at
University of Michigan, Oxford and Warwick, and at the First Asian
Quantitative Finance Conference (Singapore, January 2013), the Risk
and Stochastics Conference, (LSE, May 2013), the Workshop on New
Development in Stochastic Analysis: Probability and PDE
interactions, (Beijing, July 2013), and the INFORMS Annual Meeting
2013, (Minneapolis, October 2013).


\begin{thebibliography}{10}
\bibitem{Bayraktar}
Bayraktar, E. and Song, Y., Quadratic reflected BSDEs with unbounded
obstacles, {\it Stochastic Processes and Their Applications}, 122,
(2012), 1155--1203.

\bibitem{Chassagneux0}
Bouchard, B. and Chassagneux, J.-F., Discrete-time approximation for
continuously and discretely reflected BSDEs, {\it Stochastic
Processes and their Applications}, 118(12), (2008), 2269--2293.

\bibitem{Chassagneux1}
Chassagneux, J.-F., An introduction to the numerical approximation
of BSDEs, {\it Lecture notes of second school of CREMMA}, (2012).

\bibitem{ELIE}
Chassagneux, J.-F., Elie, R. and Kharroubi, I., A note on existence
and uniqueness for solutions of multidimensional reflected BSDEs,
{\it Electronic Communications in Probability}, 16, (2011),
120--128.


\bibitem{Cvitanic}
Cvitani{\'c}, J., Karatzas, I. and Soner, H. M., {Backward
stochastic differential equations with constraints on the
gains-process}, {\it The Annals of Probability}, {26(4)}, (1998),
{1522--1551}.

\bibitem{Dai}
Dai, M., Kwok, Y. and You, H., Intensity-based framework and penalty
formulation of optimal stopping problems, {\it Journal of Economic
Dynamics and Control}, 31(12), (2007), 3860-3880.

\bibitem{Dupuis2002}
Dupuis, P. and Wang, H., Optimal stopping with random intervention
times, {\it Adv. in Appl. Probab.}, 34(1), (2002), 141-157.

\bibitem{ElKaroui19971}
El Karoui, N., Kapoudjian, C., Pardoux, E., Peng, S. and Quenez, M.
C., Reflected solutions of backward {SDE}s, and related obstacle
problems for {PDE}s, {\it Ann. Probab.}, 25(2), (1997), 702-737.

\bibitem{ElKaroui19972}
El Karoui, N., Pardoux, E. and Quenez, M. C., Reflected backward
{SDE}s and {A}merican options, {\it Numerical methods in finance,
Publ. Newton Inst., Cambridge Univ. Press}, 13, (1997), 215-231.

\bibitem{ElKaroui19973}
El Karoui, N., Peng, S., Quenez, M., Backward SDEs in finance, {\it
Math. Finance}, 7(1), (1997), 1--71.

\bibitem{Gyongy}
Gy\"ongy, I. and Siska, D., On randomized stopping, {\it Bernoulli},
14(2), (2008), 352--361.

\bibitem{Hamadene1}
Hamad\`ene, S. and Jeanblanc, M., On the starting and stopping
problem: Application in reversible investments, {\it Math. Oper.
Res.}, 32 (1), (2007), 182--192.

\bibitem{Hamadene2}
Hamad\`ene, S. and Zhang, J., Switching problem and related system
of reflected backward SDEs, {\it Stochastic Processes and Their
Applications}, 120, (2010), 403--426.

\bibitem{Hu}
Hu, Y. and Tang, S., Multi-dimensional BSDE with oblique reflection
and optimal switching, {\it Probab. Theory Related Fields}, 147,
(2010), 89--121.

\bibitem{Kobylanski}
Kobylanski, M., Lepeltier, J. P., Quenez, M. C. and Torres, S.,
Reflected BSDE with superlinear quadratic coefficient, {\it
Probability and Mathematical Statistics}, 22(1), (2002), 51--83.

\bibitem{Krylov}
Krylov, N.V., Controlled diffusion processes. {\it Springer},
(2008), 2nd printing edition.

\bibitem{Lempa}
Lempa, J., Optimal stopping with information constraint, {\it
Applied Mathematics and Optimization}, 66(2), (2012): 147--173.

\bibitem{Pham}
Ly Vath, V., and Pham, H.,
Explicit solution to an optimal switching problem in the two-regime case,
{\it SIAM Journal on Control and Optimization}, 46(2), (2007), 395–-426.

\bibitem{Xu1}
Lepeltier, J.-P. and Xu, M., Penalization method for reflected
backward stochastic differential equations with one RCLL barrier,
{\it Statistics and Probability Letters}, 75, (2005), 58-66.

\bibitem{Xu2}
Lepeltier, J.-P., Xu, M., Reflected backward stochastic differential
equations with two RCLL barriers, {\it ESAIM: Probability and
Statistics}, 11, (2007), 3-22.

\bibitem{Liang}
Liang, G., L\"utkebohmert, E. and Wei, W., Funding liquidity, debt
tenor structure, and creditor's belief: an exogenous dynamic debt
run model, {\it Mathematics and Financial Economics}, to appear.

\bibitem{LiangWei}
Liang, G. and Wei, W., Optimal switching at Poisson random
intervention times, {\it Discrete and Continuous Dynamical
Systems-Series B}, to appear.


\bibitem{Ma}
Ma, J. and Zhang, J., Representations and regularities for solutions
to BSDEs with reflections, {\it Stochastic processes and their
applications}, 115, (2005), 539--569.

\bibitem{Peng}
Peng, S., Monotonic limit theory of BSDE and nonlinear decomposition
theorem of Doob-Meyer's type, {\it Probability Theory and Related
Fields}, 113, (1999), 473--499.

\bibitem{Xu}
Peng, S. and Xu, M., Smallest g-supermartingale and related
reflected BSDE with single and double L2 barriers, {\it Annales of
I'Institut de H. Poincare,} 41, (2005), 605--630.

\bibitem{Xu22}
Peng, S. and Xu, M., {Reflected {BSDE} with a constraint and its
applications in an incomplete market}, {\it Bernoulli}, {16(3)},
{(2010)}, {614--640}.

\bibitem{Qian}
Qian, Z. and Xu, M., Skorohod equation and reflected backward
stochastic differential equations, {\it Preprint}.

\bibitem{Stokey}
Stokey, N. L., The economics of inaction: Stochastic control models
with fixed costs, {\it Princeton University Press}, (2008).

\bibitem{Wu}
Tao, R., Wu, Z., and Zhang, Q., BSDEs with regime switching: Weak
convergence and applications, {\it Journal of Mathematical Analysis
and Applications}, 407(1), (2013), 97--111.

\bibitem{Zhang1}
Yin, G., Zhang, Q., Continuous-Time Markov Chains and Applications:
A Two-Time-Scale Approach, {\it Springer-Verlag, New York}, (2012).

\bibitem{Zhang2}
Yin, G., Zhu, C., Hybrid Switching Diffusions: Properties and
Applications, {\it Springer-Verlag, New York,} (2010).

\end{thebibliography}
\end{document}